\documentclass[oneeqnum]{article}
\usepackage{amsmath,amscd,amssymb,amsfonts,epsf,graphicx}
\usepackage{amsthm,subfigure}
\usepackage{afterpage}

\newtheorem{theorem}{Theorem}[section]

\newtheorem{lemma}{Lemma}[section]

\newcommand{\R}{{\mathbb R}}
\newcommand{\C}{{\mathbb C}}

\newcommand{\Om}{\Omega}
\newcommand{\DOm}{\partial\Omega}
\newcommand{\dbar}{\bar{\partial}}
\newcommand{\dd}{\partial}

\DeclareMathOperator{\de}{\partial}
\DeclareMathOperator{\dez}{\de_z}
\DeclareMathOperator{\dbarz}{\dbar_z}

\title{A direct D-bar reconstruction algorithm for recovering a complex conductivity in 2-D}

\begin{document}

\author{S. J. Hamilton\thanks{Department of Mathematics, Colorado State University, USA}, C. N. L. Herrera\thanks{Department of Mechanical Engineering, University of S\~ao Paulo, Brazil}, J. L. Mueller\thanks{Department of Mathematics and School of Biomedical Engineering, Colorado State University, USA}, and A. Von Herrmann\thanks{Department of Mathematics and Statistics, Colby College, USA}}

\date{August 1, 2012}
 \maketitle

\begin{abstract}
\noindent A direct reconstruction algorithm for complex conductivities in $W^{2,\infty}(\Om)$, where $\Om$ is a bounded, simply connected Lipschitz domain in $\R^2$, is presented.  The framework is based on the uniqueness proof by Francini [Inverse Problems  {\bf 20} 2000], but equations relating the Dirichlet-to-Neumann to the scattering transform and the exponentially growing solutions are not present in that work, and are derived here.   The algorithm constitutes the first D-bar method for the reconstruction of conductivities and permittivities in two dimensions.  Reconstructions of numerically simulated chest phantoms with discontinuities at the organ boundaries are included.
\end{abstract}

%\begin{keyword}
%inverse problem \sep  inverse conductivity problem \sep  complex geometrical optics solution \sep  nonlinear Fourier transform \sep electrical impedance tomography
%\MSC
%\end{keyword}

\section{Introduction}

The reconstruction of admittivies $\gamma$ from electrical boundary measurements is known as the {\em inverse admittivity problem}.  The unknown admittivity appears as a complex coefficient $\gamma(z) = \sigma(z) +i\omega\epsilon(z)$  in the generalized Laplace equation
\begin{equation} \label{genLap}
\nabla\cdot(\gamma(z)\nabla u(z) ) = 0, \quad z\in\Omega, \qquad u|_{\partial\Omega} = f,
\end{equation}
where $u$ is the electric potential, $\sigma$ is the conductivity of the medium, $\epsilon$ is the permittivity, and $\omega$ is the temporal angular frequency of the applied electromagnetic wave.  The data is the Dirichlet-to-Neumann, or voltage-to-current density  map  defined by
\begin{equation} \label{DNmap}
  \Lambda_{\gamma}u = \left.\gamma\frac{\partial u}{\partial\nu}\right|_{\DOm},
\end{equation}
where $u\in H^1(\Om)$ is the solution to \eqref{genLap}. By the trace theorem
$\Lambda_\gamma:  H^{1/2}(\DOm)\longrightarrow H^{-1/2}(\DOm)$.

In this work we present a direct reconstruction algorithm for the admittivity~$\gamma$.  The majority of the theory is based on the 2000 paper by Francini \cite{francini} in which it is established that if $\sigma,\; \epsilon \in W^{2,\infty}(\Omega)$, where $\Omega$ is a bounded domain in $\R^2$ with Lipschitz boundary, then the real-valued functions $\sigma$ and $\epsilon$  are uniquely determined by the Dirichlet-to-Neumann map, provided that the imaginary part of the admittivity is sufficiently small.   The proof in \cite{francini} is based on the D-bar method and is nearly constructive, but equations linking the scattering transform and the exponentially growing solutions to the Dirichlet-to-Neumann data are not used in the proof, and so it does not contain a complete set of equations for reconstructing the admittivity.  In this work, we derive the necessary equations for a direct, nonlinear reconstruction algorithm for the admittivity~$\gamma$.  Furthermore, we establish the existence of exponentially growing solutions to \eqref{genLap}, which prove to be useful in relating the Dirichlet-to-Neumann data to the scattering transform.  The reconstruction formula in \cite{francini} is for the potential~$Q_\gamma$, whose relationship to~$\gamma$ is described below.  We provide a direct formula for~$\gamma$ from the D-bar equations in \cite{francini}, which is computationally advantageous as well.

The inverse admittivity problem has an important application known as electrical impedance tomography (EIT).  The fact that the electrical conductivity and permittivity vary in the different tissues and organs in the body allows one to form an image from the reconstructed admittivity distribution.  In the 2-D geometry, EIT is clinically useful for chest imaging.  Conductivity images have been used for monitoring pulmonary perfusion \cite{BrownBarber, frerichs_perfusion,smit2}, determining  regional ventilation in the lungs \cite{frerichs2007, frerichs2002, victorino}, and the detection of pneumothorax \cite{costa}, for example.  In three dimensions, conductivity images have been used, for instance, in head imaging \cite{TGBH2001,TGBH2001Jutta} and knowledge of the admittivity has been applied to breast cancer detection \cite{boverman, kao,kerner}.

Reconstruction algorithms based on a least-squares approach that reconstruct permittivity include \cite{edic, boverman,jain}. The aforementioned algorithms are iterative, whereas the work presented here is a direct method that makes use of exponentially growing solutions, or complex geometrical optics ({\sc CGO}) solutions, to the admittivity equation.  The steps of the algorithm are to compute these {\sc CGO} solutions from knowledge of the Dirichlet-to-Neumann map, to compute a scattering transform matrix, to solve two systems of $\dbar$ (D-bar) equations in the complex frequency variable $k$ for the {\sc CGO} solutions to a related elliptic system, and finally to reconstruct the admittivity from the values of these solutions at $k=0$.  In this work, we provide a complete implementation of this algorithm and present reconstructions of several numerical phantoms relevant to medical EIT imaging.  The phantoms we consider here are discontinuous at the organ boundaries, which is actually outside the theory of the algorithm.  The work \cite{sarah_thesis} contains computations of smooth admittivities and validates our formulas and computations by comparing the results of the intermediate functions ({\sc CGO} solutions and scattering transforms) with those computed from knowledge of the admittivity.
% Although most EIT algorithms reconstruct only the conductivity, physiological features can be visible in the permittivity image that are not visible in the conductivity image.  This phenomenon is often frequency dependent (that is, it depends on $\omega$.)

We briefly review the history of results using {\sc CGO} solutions on the inverse conductivity problem in dimension 2.  The inverse conductivity problem was first introduced by A.P. Calder{\'o}n \cite{calderon} in 1980, where he proved that, in a linearized version of the problem, the Dirichlet-to-Neumann map uniquely determines the conductivity, and he proposed a direct reconstruction method for this case. An implementation in dimension two for experimental data is found in \cite{bikowski}. In 1996, Nachman \cite{nachman} presented a constructive proof of global uniqueness for twice differentiable conductivities using D-bar methods.  The D-bar algorithm following from \cite{nachman,siltanen} has been applied to simulated data in \cite{properties,siam,tmireview,AIP} and to experimental data on tanks and {\em in vivo} human data in \cite{TMIdata,ChestPaper,murphy1,deangelo}.  While the initial scattering transform was regularized using a Born approximation, a more recent paper \cite{dbar_regul} contains a full nonlinear regularization analysis, including estimates on speed of convergence in Banach spaces, for twice differentiable conductivities.   The regularity conditions on the conductivity were relaxed to once-differentiable in  \cite{brownuhlmann}.  The proof uses D-bar techniques and formulates the problem as a first-order elliptic system.  A reconstruction method based on \cite{brownuhlmann} can be found in \cite{knudsen,knudsenMummy,knudsenTamasan}.  Francini \cite{francini} provided a proof of unique identifiability for the inverse admittivity problem for $\sigma,\; \epsilon
 \in W^{2,\infty}(\Omega)$, with $\omega$ small. Her work provides a nearly constructive proof based on D-bar methods on a first-order elliptic system similar to that in \cite{brownuhlmann}.  A non-constructive proof that applies to complex admittivities with no smallness assumption is found in \cite{Bukhgeim2007}.  Astala and P\"aiv\"arinta provide a {\sc CGO}-based constructive proof for real conductivities $\sigma \in L^{\infty }(\Omega ),$ and numerical results related to this work can be found in \cite{APpaper1,APpaper2}.

The paper is organized as follows.  In Section \ref{sec:theory} we describe the direct reconstruction algorithm, which is comprised of boundary integral equations for the exponentially growing solutions to \eqref{genLap} involving the Dirichlet-to-Neumann data, boundary integral equations relating those {\sc CGO} solutions and the {\sc CGO} solutions $\Psi$ of the first order system, equations for the scattering transform involving only the traces of $\Psi$, the $\bar{\partial}_k$ equations established in \cite{francini}, and the direct reconstruction formula for $Q_\gamma$ and thus $\gamma$. Derivations of the novel equations are found in this section.  Section \ref{sec:numerics} describes the numerical implementation of the algorithm.  Results on noisy and non-noisy simulated data of a cross-sectional chest with discontinuous organ boundaries are found in Section \ref{sec:results}.

\section{The Direct Reconstruction Algorithm} \label{sec:theory}

In this section we will provide the equations for the direct reconstruction algorithm, completing the steps for the proof in \cite{francini} to be completely constructive.  In particular, boundary integral equations relating the {\sc CGO} solutions to the Dirichlet-to-Neumann (DN) map are derived.

Let $\Omega\subset\R^2$ be a bounded open domain with a Lipschitz boundary. Throughout we assume that there exist positive constants $\sigma_0$ and $\beta$ such that \begin{equation}\label{eq-fran-1-5}
\sigma(z)>\sigma_0,\quad z\in\Omega\subset\R^2
\end{equation}
and
\begin{equation}\label{eq-fran-1-7}
\|\sigma\|_{W^{1,\infty}(\Omega)},\;\|\epsilon\|_{W^{1,\infty}(\Omega)}\leq \beta.
\end{equation}
We extend $\sigma$ and $\epsilon$ from $\Omega$ to all of $\R^2$ such that $\sigma\equiv1$ and $\epsilon\equiv0$ outside a ball with fixed radius that contains $\Omega$, and \eqref{eq-fran-1-5} and \eqref{eq-fran-1-7} hold for all of $\R^2$.  In fact, all that is required is that $\gamma$ is constant outside that ball of fixed radius; for convenience we look at the case where $\gamma\equiv1$.

The proof in \cite{francini} closely follows that of \cite{brownuhlmann} for conductivities $\sigma \in W^{1,p}(\Omega)$,  $p>2$.   The matrix potential $Q_{\gamma}$ is, however, defined slightly differently, and since the potential in \cite{francini} is not Hermitian, the approach in \cite{francini} is to consider the complex case as a perturbation from the real case provided the imaginary part of $\gamma$ is small.  Define $Q_{\gamma}(z)$ and a matrix operator D by
\begin{equation}
\label{matrixim}
\begin{split}
Q_{\gamma}(z) = \begin{pmatrix} 0 & -\frac{1}{2}\partial_z\log\,\gamma(z) \\{-\frac{1}{2}\bar{\partial}_z\log \,\gamma(z)} & 0\end{pmatrix},
\qquad D = \begin{pmatrix} \bar\partial_z &  0\\0 & \partial_z\end{pmatrix}.
\end{split}
\end{equation}
Thus we define
\begin{equation}\label{Q_log}
Q_{12}(z)=-\frac{1}{2}\partial_z\log\,\gamma(z) \quad \mbox{and} \quad Q_{21}(z)=-\frac{1}{2}\dbar_z \log\gamma(z),
\end{equation}
and equivalently we can write
\begin{equation}\label{q_frac}
Q_{12}(z)=-\frac{\partial_z \gamma^{1/2}(z)}{\gamma^{1/2}(z)} \quad \mbox{and} \quad Q_{21}(z)=-\frac{\bar{\partial}_z \gamma^{1/2}(z)}{\gamma^{1/2}(z)},
\end{equation}
or
\begin{equation}\label{q_frac2}
Q_{12}(z)=-\frac{1}{2}\gamma^{-1}(z)\partial \gamma(z) \quad \mbox{and} \quad Q_{21}(z)=-\frac{1}{2}\gamma^{-1}(z)\bar {\partial} \gamma(z).
\end{equation}
Defining a vector
\begin{equation}\label{vw}
\begin{pmatrix} {v}\\{w}\end{pmatrix}=\gamma^{1/2}\begin{pmatrix} \partial {u}\\\bar\partial{u} \end{pmatrix},
\end{equation}
in terms of the solution $u$ to \eqref{genLap}, one sees that
\begin{equation*}
D\begin{pmatrix} {v}\\{w}\end{pmatrix} -Q_{\gamma}\begin{pmatrix} {v}\\{w}\end{pmatrix}=0.
\end{equation*}

The uniqueness result in \cite{francini} is
\begin{theorem} (Theorem 1.1 \cite{francini})  Let $\Om$ be an open bounded domain in $\R^2$ with Lipschitz boundary.  Let $\sigma_j$ and $\epsilon_j$, for $j=1,2$ satisfy assumptions \eqref{eq-fran-1-5} and $\|\sigma\|_{W^{2,\infty}(\Om)},\;\|\epsilon\|_{W^{2,\infty}(\Om)}\leq \beta$.  There exists a constant $\omega_0=\omega_0(\beta,\sigma_0,\Om)$ such that if $\gamma_j=\sigma_j + i\omega\epsilon_j$ for $j=1,2$ and $\omega<\omega_0$ and if
$$\Lambda_{\gamma_1} = \Lambda_{\gamma_2},$$
then
$$\sigma_1=\sigma_2 \quad \mbox{and} \quad \epsilon_1=\epsilon_2.$$
\end{theorem}

\subsection{{\sc CGO} solutions}
Francini shows in \cite{francini} that for $\omega$ sufficiently small and $\gamma$ satisfying \eqref{eq-fran-1-5} and \eqref{eq-fran-1-7}  there exists a unique $2\times2$ matrix $M(z,k)$ for $k\in\C$ satisfying
\begin{equation} \label{Mcondition}
M(\cdot,k)-I \in L^p(\R^2), \quad \mbox{for some} \quad p>2,
\end{equation}
that is a solution to
\begin{equation} \label{Dbar_z_for_M}
(D_{k}-Q_{\gamma}(z))M(z,k) = 0,
\end{equation}
where $D_{k}$ is the matrix operator defined by
\[D_{k}M=DM-ik\begin{pmatrix} 1&  0 \\  0 & -1 \end{pmatrix}M^{\mbox{off}},\]
and ``off'' denotes the matrix consisting of only the off-diagonal entries of $M$.
The system \eqref{Dbar_z_for_M} is equivalent to the following set of equations, included for the reader's convenience
\begin{equation}\label{M_eqs}
\begin{array}{rclcrcl}
\dbar_z M_{11} - Q_{12}M_{21} &=&0&\qquad& (\dbar_z-ik)M_{12} - Q_{12}M_{22}&=&0\\
(\dd_z+ik)M_{21} - Q_{21}M_{11} &=&0&\qquad&\dd_z M_{22} - Q_{21}M_{12} &=&0.
\end{array}
\end{equation}
%\begin{eqnarray} \label{M_eqs}
%\dbar_z M_{11} - Q_{12}M_{21} &=& 0 \qquad (\dd_z+ik)M_{21} - Q_{21}M_{11} = 0 \\
%(\dbar_z-ik)M_{12} - Q_{12}M_{22} &=& 0 \qquad \dd_z M_{22} - Q_{21}M_{12} =0.  \nonumber
%\end{eqnarray}
Thus, there exists a unique matrix $\Psi(z,k)$ defined by
\begin{equation}
\label{matrix21}
\Psi (z,k)=M(z,k)\begin{pmatrix} e^{izk} & 0\\0 & e^{-i\bar{z}k}\end{pmatrix}=\begin{pmatrix} e^{izk}M_{11}(z,k) & e^{-i\bar{z}k}M_{12}(z,k) \\ e^{izk}M_{21}(z,k) & e^{-i\bar{z}k}M_{22}(z,k) \end{pmatrix},
\end{equation}
that is a solution to
\begin{equation}
\label{matrix20}
\begin{pmatrix} {D-Q_{\gamma}}\end{pmatrix}\Psi=0,
\end{equation}
or equivalently
\begin{eqnarray} \label{psi_eqs}
\dbar_z\Psi_{11} - Q_{12}\Psi_{21} &=& 0 \qquad \dd_z\Psi_{21} - Q_{21}\Psi_{11} = 0 \\
\dbar_z\Psi_{12} - Q_{12}\Psi_{22} &=& 0 \qquad \dd_z\Psi_{22} - Q_{21}\Psi_{12} =0.  \nonumber
\end{eqnarray}

These {\sc CGO} solutions $\Psi(z,k)$ are key functions in the reconstructions, but the proof in \cite{francini} does not provide a link from these functions to the DN data.  A useful link can be established through exponentially growing solutions to the admittivity equation \eqref{genLap}.  For $\gamma-1$ with compact support, equation \eqref{genLap} can be studied on all of $\R^2$, and introducing the complex parameter $k$, two distinct exponentially growing solutions, which differ in their asymptotics, exist.  We will denote these solutions by $u_1$ and $u_2$ where $u_1\sim \frac{e^{ikz}}{ik}$ and $u_2\sim \frac{e^{-ik\bar{z}}}{-ik}$ in a sense that is made precise in Theorems~\ref{thm:uexp} and \ref{thm:uexp2}, where the existence of such solutions is established.  The proof will make use of the following lemma proved in the real case by Nachman \cite{nachman}; the complex version shown here also holds and was used in \cite{francini}. The lemma is also true if $\dbarz$ is interchanged with $\dez$.
\begin{lemma}  \label{N}
Let $1<s<2$ and $\frac{1}{r}=\frac{1}{s} -\frac{1}{2}.$
\begin{enumerate}
\item If the complex function $v\in L^{s}(\mathbb{R}^2),$ then there exists a unique complex function $u\in L^{r}(\mathbb{R}^2)$ such that $(\partial_z+ik)u=v.$
\item If the complex function $v\in L^{r}(\mathbb{R}^2)$ and ${\bar{\partial}_z v} \in L^{s}(\mathbb{R}^2)$, $k \in \mathbb{C}\setminus\{0\},$ then there exists a unique complex function $u\in W^{1,r}(\mathbb{R}^2)$ such that $(\partial_z+ik)u=v.$
\item If the complex function $v\in L^{r}(\mathbb{R}^2)$ and ${\bar{\partial}_z v} \in L^{s}(\mathbb{R}^2)$, $k \in \mathbb{C}\setminus\{0\},$ then there exists a unique complex function $u\in W^{1,r}(\mathbb{R}^2)$ such that $(\bar\partial_z-ik)u=v.$
\end{enumerate}
\end{lemma}

The following lemma will also be used in the proofs of Theorems~\ref{thm:uexp} and \ref{thm:uexp2}.
\begin{lemma}
For $\omega$ sufficiently small and $\gamma$ satisfying \eqref{eq-fran-1-5} and \eqref{eq-fran-1-7}, the following identities hold:
\begin{eqnarray}
\dbarz({\gamma}(z)^{-1/2}M_{11}(z,k)-1)&=&(\dez+ik)({\gamma}(z)^{-1/2}M_{21}(z,k))\label{help}\\
\partial_z({\gamma}(z)^{-1/2}M_{22}(z,k)-1)&=&(\bar\partial_z-ik)({\gamma}(z)^{-1/2}M_{12}(z,k))\label{help1}.
\end{eqnarray}
\end{lemma}
\begin{proof}
By the product rule,
\begin{equation*}
\begin{split}
 \dbarz\left({\gamma}(z)^{-1/2}M_{11}(z,k)-1\right)
&= \dbarz\left({\gamma}(z)^{-1/2}\right)M_{11}(z,k)+ {\gamma}(z)^{-1/2} \dbarz(M_{11}(z,k))\\
&={\gamma}(z)^{-1/2}Q_{21}(z){M_{11}(z,k)}+  {\gamma}(z)^{-1/2}Q_{12}(z)M_{21}(z,k) \\
&={\gamma}(z)^{-1/2}(\dez +ik){M_{21}(z,k)}+  {\gamma}(z)^{-1/2}Q_{12}(z)M_{21}(z,k).
\end{split}
\end{equation*}
The second and third equalities utilized \eqref{Q_log} and \eqref{M_eqs}, respectively.

We also have
\begin{equation*}
\begin{split}
(\dez+ik)\left({\gamma}(z)^{-1/2}M_{21}(z,k)\right)
&= \dez\left({\gamma}(z)^{-1/2}M_{21}(z,k)\right)+ ik{\gamma}(z)^{-1/2}M_{21}(z,k)\\
&= \dez\left({\gamma}(z)^{-1/2}\right) M_{21}(z,k)+  {\gamma}(z)^{-1/2}\dez(M_{21}(z,k))\\
& \qquad+ik{\gamma}(z)^{-1/2}M_{21}(z,k) \\
&= {\gamma}(z)^{-1/2}Q_{12}(z)M_{21}(z,k) +{\gamma}(z)^{-1/2}(\dez +ik)M_{21}(z,k).
\end{split}
\end{equation*}
This establishes \eqref{help}.

Similarly, using \eqref{Q_log} and \eqref{M_eqs},
\begin{equation*}
\begin{split}
{\dez}\left({\gamma}(z)^{-1/2}M_{22}(z,k)-1\right)
&= \dez\left({\gamma}(z)^{-1/2}\right)M_{22}(z,k)+ {\gamma}(z)^{-1/2} \dez(M_{22}(z,k))\\
&= {\gamma}(z)^{-1/2}Q_{12}(z){M_{22}(z,k)}+  {\gamma}(z)^{-1/2}Q_{21}(z)M_{12}(z,k) \\
&= {\gamma}(z)^{-1/2}(\dbarz -ik){M_{12}(z,k)}+  {\gamma}(z)^{-1/2} Q_{21}(z)M_{12}(z,k).
\end{split}
\end{equation*}

We also have
\begin{equation*}
\begin{split}
(\dbarz-ik)\left({\gamma}(z)^{-1/2}M_{12}(z,k)\right)
&= {\dbarz}\left({\gamma}(z)^{-1/2}M_{12}(z,k)\right)- ik{\gamma}(z)^{-1/2}M_{12}(z,k)\\
&= \dbarz\left({\gamma}(z)^{-1/2}\right) M_{12}(z,k)+  {\gamma}(z)^{-1/2}\dbarz(M_{12}(z,k))\\
& \qquad-ik{\gamma}(z)^{-1/2}M_{12}(z,k) \\
&= {\gamma}(z)^{-1/2} Q_{21}(z)M_{12}(z,k) +{\gamma}(z)^{-1/2}(\dbarz -ik)M_{12}(z,k).
\end{split}
\end{equation*}
This establishes \eqref{help1}.
\end{proof}

Knudsen establishes the existence of exponentially growing solutions to the conductivity equation in the context of the inverse conductivity problem in \cite{knudsen}.  The proofs of their existence for the admittivity equation and the associated boundary integral equations are in the same spirit as \cite{knudsen}.

\begin{theorem} \label{thm:uexp}
Let $\gamma(z)\in W^{1,p}(\Omega),$ with $p>2$ such that $\sigma$ and $\epsilon$ satisfy \eqref{eq-fran-1-5} and \eqref{eq-fran-1-7}, and let $\gamma(z) - 1$ have compact support in $W ^{1,p}(\Omega).$ Then for all $k \in\mathbb{C}\setminus\{0\}$ there exists a unique solution
\begin{equation} \label{uw}
u_1(z,k)=e^{ikz}\left[\frac{1}{ik} + w_1(z,k)\right],
\end{equation}
to the admittivity equation in $\mathbb{R}^2$ such that $w_1(\cdot,k)\in W^{1,r}(\mathbb{R}^2),$ $2<r<\infty.$  Moreover, the following equalities hold:
\begin{equation}  \label{e1}
(\dez +ik)\left[e^{-ikz}u_1(z,k) - \frac{1}{ik}\right]= {\gamma}^{-1/2}(z)M_{11}(z,k) -1
\end{equation}
\begin{equation}
\label{e2}
\dbarz\left[e^{-ikz}u_1(x,k)- \frac{1}{ik}\right]= {\gamma}^{-1/2}(z)M_{21}(z,k),
\end{equation}
and
\begin{equation}
\label{e3}
\left\|e^{-ikz}u_1(x,k) - \frac{1}{ik}\right\|_{W^{1,r}(\mathbb{R}^2)}\leq C\left(1+\frac{1}{|k|}\right),
\end{equation}
for some constant C.
\end{theorem}

\begin{theorem} \label{thm:uexp2}
Let $\gamma(z)$ satisfy the hypotheses of Theorem~\ref{thm:uexp}.  Then for all $k \in\mathbb{C}\setminus\{0\}$ there exists a unique solution
\begin{equation} \label{uw_tilde}
u_2(z,k)=e^{-ik\bar z}\left[\frac{1}{-ik}+w_2(z,k)\right],
\end{equation}
 to the admittivity equation  in $\mathbb{R}^2$ with $w_2(\cdot,k)\in W^{1,r}(\mathbb{R}^2),$ $2<r<\infty.$ Moreover, the following equalities hold:
\begin{equation}
\label{e4}
(\dbarz -ik)\left[e^{ik\bar z}u_2(z,k) + \frac{1}{ik}\right]= {\gamma}^{-1/2}(z)M_{22}(z,k) -1
\end{equation}
\begin{equation}
\label{e5}
\dez\left[e^{ik\bar z}u_2(z,k) + \frac{1}{ik}\right]= {\gamma}^{-1/2}(z)M_{12}(z,k),
\end{equation}
and
\begin{equation}
\label{e6}
\left\|e^{ik\bar z}u_2(z,k) + \frac{1}{ik}\right\|_{W^{1,r}(\mathbb{R}^2)}\leq C\left(1+\frac{1}{|k|}\right),
\end{equation}
for some constant C.
\end{theorem}

We will prove Theorem~\ref{thm:uexp}; the proof of Theorem~\ref{thm:uexp2} is analogous.
\begin{proof}
Assume $u$ is a solution of the admittivity equation of the form \eqref{uw}, and let $(v,w)^T = \gamma^{1/2}(\dez u,\dbarz u)^T$ be the corresponding solution to $(D-Q_\gamma)\Psi = 0$.  Define the complex function $v$ via $v(z,k)={\gamma}(z)^{-1/2}M_{11}(z,k)-1.$ We will first show there exists a unique complex function $w_1\in W^{1,r}(\mathbb{R}^2)$, where $r>2$ such that $(\dez+ik)w=v$, for $k \in\mathbb{C}\setminus\{0\}.$ Let us rewrite $v$ as follows:
\[v(z,k)=\left.\gamma(z)^{-1/2}\middle[M_{11}(z,k)-1\right]+\left[{\gamma}(z)^{-1/2}-1\right].\]
Let $r>2$ and $1<s<2$ with $\frac{1}{r}=\frac{1}{s} -\frac{1}{2}.$ We know by Theorem 4.1 of \cite{francini} that there exists a constant $C>0$ depending on $\beta$, $\sigma_0$ and $p$ such that $\sup \left\|M_{11}(z,k)-1 \right\|_{L^{r}(\mathbb{R}^2)}\leq C$ for every $r>2$, and that ${\gamma}(z)^{-1/2}-1$ has compact support in $W^{1,r}(\mathbb{R}^2).$ It follows that $v\in L^r(\mathbb{R}^2)$, and by Minkowski's Inequality
\[\left\|v(z,k)\right\|_{L^{r}}=\left\|\left.\gamma(z)^{-1/2}\middle[M_{11}(z,k)-1\right]+\left[{\gamma}(z)^{-1/2}-1\right]\right\|_{L^r}\leq C_{r, \gamma},\]
where $C_{r, \gamma}$ depends on $r$ and the bounds on $\sigma$ and $\epsilon$.

From \eqref{q_frac2},
\begin{eqnarray*}
\dbarz v(z,k)  & =& \dbarz({\gamma}(z)^{-1/2}M_{11}(z,k)-1) \\
&=& (\dbarz{\gamma}(z)^{-1/2})M_{11}(z,k)+ {\gamma}(z)^{-1/2} (\dbarz M_{11}(z,k)) \\
&= &{\gamma}(z)^{-1/2}Q_{21}(z)M_{11}(z,k)+  {\gamma}(z)^{-1/2}Q_{12}(z)M_{21}(z,k) \\
&= &{\gamma}(z)^{-1/2}Q_{21}(z)[M_{11}(z,k)-1] + \gamma(z)^{-1/2}Q_{21}(z)\\
&&\quad +\;\gamma(z)^{-1/2}Q_{12}(z)M_{21}(z,k)
\end{eqnarray*}
%The third equality came from \eqref{q_frac2} and in the last line we added and subtracted $\gamma^{-1/2}Q_{21}$.

We know that ${\gamma}(z)^{-1/2}Q_{21}(z)  \in L^{\alpha}(\mathbb{R}^2)$ with $1\leq {\alpha}\leq {p}$	since $Q_{12}(z)$ has compact support. It follows that ${\gamma}(z)^{-1/2}Q_{21}(z)\in L^{s}(\mathbb{R}^2)\cap L^{2}(\mathbb{R}^2).$ By the generalized H\"older's inequality and the fact that $\left\| M_{11}(z,k)-1\right\|_{L^{s}}$ is bounded with $\frac{1}{s}=\frac{1}{r}+\frac{1}{2},$ we have $\dbarz v\in L^{s}(\mathbb{R}^2)$ and
$\left\|\dbarz v\right\|_{L^{s}(\mathbb{R}^2)} \leq K_{r,\gamma},$
where $K_{r,\gamma}$ depends only on $r$ and the bounds on $\sigma$ and $\epsilon$. Thus, by Lemma \ref{N} (2), there exists a unique solution $w_1(z,k)\in W^{1,r}(\mathbb{R}^2)$ such that
\begin{equation}
\label{construct}
(\dez +ik)w_1(z,k)={\gamma}(z)^{-1/2}M_{11}(z,k)-1.
\end{equation}
We have by \eqref{help},
\begin{equation}
\label{yeah}
\bar{\partial}({\gamma}(z)^{-1/2}M_{11}(z,k)-1)=\left(\dez +ik\right)\left(\gamma(z)^{-1/2}M_{21}(z,k)\right).
\end{equation}
Taking $\dbarz$ of both sides of \eqref{construct} and using \eqref{yeah},
\begin{equation}
\begin{split}
\dbarz\left(\dez+ik\right)w_1(z,k)
& = \dbarz\left({\gamma}(z)^{-1/2}M_{11}(z,k)-1\right)\\
&= \left(\dez+ik\right)\left({\gamma}(z)^{-1/2}M_{21}(z,k)\right).
\end{split}
\end{equation}
Using the fact $\bar{\partial}(\partial+ik)=(\partial+ik)\bar{\partial}$, it follows that
\begin{equation}
\label{unique3}
\left(\dez+ik\right)\left(\dbarz w_1(z,k)-{\gamma}(z)^{-1/2}M_{21}(z,k)\right)=0.
\end{equation}
Since $\dbarz w_1(z,k)-{\gamma}(z)^{-1/2}M_{21}(z,k)\in L^{r}(\mathbb{R}^2),$ by Lemma \ref{N} (1), we must have
\begin{equation}
\label{sure}
\dbarz w_1(z,k)={\gamma}(z)^{-1/2}M_{21}(z,k).
\end{equation}
We now define
\begin{equation}
\label{cu}
u_1(z,k)=e^{ikz}\left[w_1(z,k)+\frac{1}{ik}\right],
\end{equation}
then by \eqref{construct}
\[\left(\dez +ik\right)\left(e^{-ikz}u_1(z,k) - \frac{1}{ik}\right)=\left(\dez +ik\right)w_1(z,k)= {\gamma}^{-1/2}(z)M_{11}(z,k) -1,\]
which proves \eqref{e1}, and by \eqref{sure}
\[\dbarz\left(e^{-ikz}u_1(z,k)- \frac{1}{ik}\right)=\dbarz w_1(z,k)= {\gamma}^{-1/2}(z)M_{21}(z,k),\]
which proves \eqref{e2}.

The norm estimate given by \eqref{e3} follows by Minkowski's Inequality, the constant $C$ depends on $r$, the bound on $\gamma-1$, and the bounds on $\sigma$ and $\epsilon$.
\end{proof}

\noindent
\textbf{Remark:} Note that from \eqref{e1}
\begin{equation}
\begin{split}
\gamma^{-1/2}M_{11}(z,k)-1
&=(\dez+ik)\left(e^{-ikz}u_1(z,k)-\frac{1}{ik}\right)\\
&=\dez(e^{-ikz}u_1)+ike^{-ikz} u_1(z,k) -1 \\
&=e^{-ikz}\dez u_1(z,k) - 1,
\end{split}
\end{equation}
and from \eqref{e2}
\begin{equation}
\begin{split}
\gamma^{-1/2}M_{21}(z,k)
&=\dbarz\left(e^{-ikz}u_1(z,k)-\frac{1}{ik}\right)\\
&=u_1(z,k)\dbarz\left(e^{-ikz}\right)+e^{-ikz}\dbarz u_1(z,k)\\
&=e^{-ikz}\dbarz u_1(z,k).
\end{split}
\end{equation}
Thus, we can equivalently rewrite \eqref{e1} and \eqref{e2}, respectively, as
\begin{eqnarray}
\gamma^{1/2}(z)\dez u_1(z,k)&=&e^{ikz}M_{11}(z,k)=\Psi_{11}(z,k)\label{e7}\\
\gamma^{1/2}(z)\dbarz u_1(z,k)&=&e^{ikz}M_{21}(z,k)=\Psi_{21}(z,k).\label{e8}
\end{eqnarray}
In a similar manner, we can rewrite \eqref{e4} and \eqref{e5}, respectively, as
\begin{eqnarray}
\gamma^{1/2}(z)\dbarz u_2(z,k) &=& e^{-ik\bar z}M_{22}(z,k)=\Psi_{22}(z,k)  \label{e9}\\
\gamma^{1/2}(z)\dez u_2(z,k) &=& e^{-ik\bar{z}}M_{12}(z,k)=\Psi_{12}(z,k)\label{e10}.
\end{eqnarray}

Useful boundary integral equations for the traces of $u_1$ and $u_2$ can be derived under the additional assumption that $\gamma\in W^{2,p}$ and $u_1, u_2\in W^{2,p}$, $p>1$.  The following proposition shows a relationship between the exponentially growing solutions $\psi_S(z,k)$ (when they exist) to the Schr\"{o}dinger equation
\begin{equation} \label{Schr}
(-\Delta + q_S(z))\psi_S (z,k)=0,
\end{equation}
and the {\sc CGO} solutions $u_1$ and $u_2$ to \eqref{genLap}.  The solution $\psi_S$  to \eqref{Schr}, where $q_S$ is complex, is asymptotic to $e^{ikz}$ in the sense that
\[w_S\equiv e^{-ikz}\psi_S(\cdot,k)-1 \in L^{\tilde{p}}\cap L^{\infty},\]
where $\frac{1}{\widetilde p}=\frac{1}{p}-\frac{1}{2}$ and $1<p<2$.  %We will also use the fact that by the Sobolev Embedding Theorem, $W^{2,p}(\Omega)\subset W^{1,p}(\Omega),$ with $1<p<2.$
The question of the existence of a unique solution  to \eqref{Schr} is addressed for real $\gamma$ in \cite{nachman}, where it is shown to exist if and (roughly) only if $q_S=\frac{\Delta\gamma^{1/2}}{\gamma^{1/2}}.$  The solutions $\psi_S$ will be used to derive the boundary integral equations for $u_1$ and $u_2$, but not in the direct reconstruction algorithm.

%\begin{lemma} (\cite{nachman})
%\label{food}
%Let $q_{\sigma}=\sigma ^{-1/2}\Delta \sigma ^{1/2},$ with $\sigma\in W^{2,p}(\mathbb{R}^2),$ $1<p<2$. Then for any $k\in \mathbb{C}\setminus\left\{0\right\}$ there exists $\psi_S(\cdot,k)$ in
%\begin{equation}
%\label{sk}
%(-\Delta +q_{\sigma})\psi_S(z,k)=0, \,\,\, \mbox{z}\in \mathbb{R}^{2}
%\end{equation}
%with $w(\cdot,k)=\psi_S(\cdot,k)e^{-ikz}-1 \in W^{1,\widetilde{p}}(\mathbb{R}^2),$ where $\frac{1}{\widetilde p}=\frac{1}{p}-\frac{1}{2}.$
%\end{lemma}

\begin{lemma}  % Lemma for u1
Let $\gamma(z)=\sigma(z)+i\omega\epsilon(z)\in W^{2,p}(\Omega),$ with $p>2$ such that $\sigma$ and $\epsilon$ satisfy \eqref{eq-fran-1-5} and \eqref{eq-fran-1-7}, and let $\gamma(z) - 1$ have compact support in $W ^{1,p}(\Omega).$  Let $u_1$ be the exponentially growing solution to the admittivity equation as given in Theorem \ref{thm:uexp}, and let $\psi_S$ be the exponentially growing solution to the Schr\"{o}dinger equation \eqref{Schr},
when it exists. Then
\begin{equation}  \label{sa}
iku_1(z,k)={\gamma}^{-1/2}(z)\psi_S(z,k).
\end{equation}
\end{lemma}

\begin{proof}  %Proof of lemma for u1
From \eqref{uw},
\begin{equation*}
\begin{split}
iku_1(z,k)
& = e^{ikz}(1+ikw_1(z,k))\\
& = e^{ikz}{\gamma}^{-1/2}(z)\left[{\gamma}^{1/2}(z)+ {\gamma}^{1/2}(z)ikw_1(z,k)\right]\\
& = e^{ikz}{\gamma}^{-1/2}(z)\left(1+\left[{\gamma}^{1/2}(z) -1\right]+{\gamma}^{1/2}(z)ikw_1(z,k)\right)
\end{split}
\end{equation*}
satisfies the admittivity equation with $[{\gamma}^{1/2}(z)-1]+ {\gamma}^{1/2}(z) ikw_1(z,k)\in W^{1,r}(\Omega)$  for   $r>2.$  We also know that when it exists,
\begin{equation} \label{psi_S_rel}
\gamma^{-1/2}(z)\psi_S(z,k)=e^{ikz}\gamma^{-1/2}(z)(1+w_S(z,k))
\end{equation}
is also a solution to the admittivity equation with  $w_S(z,k)\in W^{1,\bar{p}}(\mathbb{R}^2).$
Hence, these exponentially growing solutions must be equal.
\end{proof}

\begin{lemma}  % Lemma for u2
Let $\gamma(z)=\sigma(z)+i\omega\epsilon(z)\in W^{2,p}(\Omega),$ with $p>2$ such that $\sigma$ and $\epsilon$ satisfy \eqref{eq-fran-1-5} and \eqref{eq-fran-1-7}, and let $\gamma(z) - 1$ have compact support in $W ^{1,p}(\Omega).$  Let $u_2$ be the exponentially growing solution to the admittivity equation as given in Theorem \ref{thm:uexp2}, and let $\psi_S$ be the exponentially growing solution to the Schr\"{o}dinger equation \eqref{Schr},
when it exists. Then
\begin{equation} \label{sa2}
-iku_2(z,k)={\gamma}^{-1/2}(-\bar{z})\psi_S(-\bar{z},k).
\end{equation}
\end{lemma}

\begin{proof}  % Proof of lemma for u2
From \eqref{uw_tilde},
\begin{equation*}
\begin{split}
-iku_2(z,k)
& = e^{-ik\bar{z}}(1-ikw_2(z,k))\\
& = e^{-ik\bar{z}}{\gamma}^{-1/2}(-\bar{z})\left(1+\left[{\gamma}^{1/2}(-\bar{z}) -1\right]-{\gamma}^{1/2}(-\bar{z})ikw_2(z,k)\right)
\end{split}
\end{equation*}
satisfies the admittivity equation with $[{\gamma}^{1/2}(-\bar{z})-1]-{\gamma}^{1/2} (-\bar{z})ikw_2(z,k)\in W^{1,r}(\Omega)$  for   $r>2.$  From \eqref{psi_S_rel},
\begin{equation*}
{\gamma}^{1/2}(-\bar{z})\psi_S(-\bar{z},k)=e^{-ik\bar{z}}{\gamma}^{1/2}(-\bar{z})(1+w_S(-\bar{z},k))
\end{equation*}
satisfies the admittivity equation with $w_S(-\bar{z},k)\in W^{1,\bar{p}}(\mathbb{R}^2).$
Thus, these exponentially growing solutions must be equal, and so
\[-iku_2(z,k)={\gamma}^{-1/2}(-\bar{z})\psi_S(-\bar{z},k).\]
\end{proof}

Let us recall some terminology arising from \cite{nachman} before establishing boundary integral equations involving the exponentially growing solutions.  Let $\Lambda_{\sigma}$ be the Dirichlet-to-Neumann map when $\Omega$  contains the conductivity distribution $\sigma$,
and $\Lambda_{1}$ is the Dirichlet-to-Neumann map for a homogeneous conductivity equal to 1. The Faddeev Green's function $G_k(z)$ is  defined by
\begin{equation}
\label{g1}
G_k(z) := e^{ik\cdot z}g_k(z), \,\,\,\, -\Delta G_k = \delta,
\end{equation}
where
\begin{equation}
\label{g2}
g_k(z) :=\frac{1}{(2\pi)^2}\int_{\mathbb{R}^2}\frac{e^{iz\cdot\xi}}
{\xi(\bar{\xi}+2k)}d\xi, \, \,\,\,\,  (-\Delta -4ik\overline{\partial})g_k = \delta,
\end{equation}
for $k\in\mathbb{C}\setminus\{0\}$.  In the  real-valued case $\gamma = \sigma$, the trace of the function $\psi_S(\cdot,k)$ on $\DOm$ satisfies the integral equation \cite{nachman}
\begin{equation}  \label{ie}
\psi_S(z,k)=e^{ikz}-\int_{\DOm}G_k(z-\zeta) (\Lambda_{\sigma}-\Lambda_{1})\psi_S(\zeta,k)dS(\zeta), \quad z\in {\partial\Omega},
\end{equation}
where ${k}\in \mathbb{C}\setminus\left\{0\right\}.$ The equation \eqref{ie} is a Fredholm equation of the second kind and uniquely solvable in $H^{1/2}(\partial\Omega)$ for any $k\in \mathbb{C}\setminus\left\{0\right\}.$

The boundary integral equations for $u_1$ and $u_2$ are similar to \eqref{ie}.
\begin{theorem}  \label{thm:u1_BIE}  % BIE for u1
Let $\gamma\in W^{2,p}(\Omega)$ for $p>1$ and suppose $\gamma=1$ in a neighborhood of $\partial\Omega$. Suppose $\sigma$ and $\epsilon$ satisfy \eqref{eq-fran-1-5} and \eqref{eq-fran-1-7}, and let $\gamma(z) - 1$ have compact
support in $W ^{2,p}(\Omega).$  Then for any nonexceptional $k\in \mathbb{C}\setminus\left\{0\right\}$, the trace of the exponentially growing solution $u_1(\cdot,k)$ on $\partial\Omega$ is the unique solution to
\begin{equation} \label{bie2}
u_1(z,k)=\frac{e^{ikz}}{ik}-\int_{\DOm}G_{k}(z-\zeta)(\Lambda_{\gamma}-\Lambda_{1})u_1(\zeta,k)dS(\zeta),\hspace{0.5 cm} z\in \partial\Omega.
\end{equation}
\end{theorem}

\begin{proof}  %Proof of BIE for u1
Let $\frac{1}{p}=\frac{1}{r}-\frac{1}{2},$ where $1<r<2$ and $p>2.$ Let $\left\{\gamma_n \right\}_{n \in\textbf{N}}\subset W^{2,r}(\Omega)$ be a sequence converging to $\gamma\in W^{1,p}(\Omega)$.  Then by the Sobolev Embedding Theorem,   $\left\{\gamma_n\right\}_{n \in\mathbb{N}}\subset W^{1,r}(\Omega).$
Let $\psi_n$ be the exponentially growing solutions to the Schr\"{o}dinger equation with potential $\gamma_n ^{-1/2}\Delta\gamma_n ^{1/2},$ and $u^n$ be the CGO solutions defined by Theorem \ref{thm:uexp} to the admittivity equation with admittivity $\gamma_n.$ Then for each $n\in \mathbb{N},$ the complex $\gamma$ version of \eqref{ie} holds for nonexceptional $k\in \mathbb{C}\setminus\left\{0\right\}$
\begin{equation}  \label{psi_N}
\psi_n(z,k)|_{\partial\Omega}=e^{ikz}|_{\partial\Omega}-\int_{\DOm}G_{k}(z-\zeta)(\Lambda_{\gamma_n}-\Lambda_{1})\psi_n(\zeta,k)dS(\zeta),
\end{equation}
where $\gamma_n=1$ in the neighborhood of $\partial\Omega$.

It follows by \eqref{sa} that for each complex number $k\neq0,$ and for each $n\in \mathbb{N}$
\begin{equation}  \label{pnun}
\frac{{\gamma_n}^{-1/2}(z)}{ik}\psi_n(z,k)=u^n(z,k) \rightarrow u_1(z,k) \hspace{0.3cm} \mbox{in} \hspace{0.4cm} H^{1/2}(\partial\Omega).
\end{equation}

We claim that for each $n,$ $u^n$ satisfies \eqref{bie2}. To see this, by \eqref{sa}, for $z\in \DOm$,
\begin{eqnarray}
\frac{e^{ikz}}{ik} &-& \int_{\DOm}G_{k}(z-\zeta)(\Lambda_{\gamma}-\Lambda_{1})u^n(\zeta,k)dS(\zeta) \nonumber\\
&=&\frac{e^{ikz}}{ik}-\int_{\DOm}G_{k}(z-\zeta)(\Lambda_{\gamma}-\Lambda_{1})\frac{\gamma_n ^{-1/2}(\zeta)}{ik} \psi_n(\zeta,k)dS(\zeta) \nonumber\\
&=&  \frac{\gamma_n ^{-1/2}(z)}{ik}\psi_n(z,k) \nonumber \\
&=&  u^n(z,k),\label{jen}
\end{eqnarray}
where we used the fact that  $\gamma_n =1$ in a neighborhood of $\partial\Omega$.  Thus, $u^n$ satisfies \eqref{bie2} for each $n\in \mathbb{N}.$

We know by Theorem~3.1 of  \cite{francini} that $M(z,k)$ depends continuously on $\gamma.$  From \eqref{pnun}, we can conclude that  % Note we are using Lemma 29 from Alan's thesis here, but don't want to write it.
\begin{equation}
\label{gotit}
\int_{\DOm}G_{k}(z-\zeta)(\Lambda_{\gamma_n}-\Lambda_{1})u^n(\zeta,k)dS(\zeta) \mapsto \int_{\DOm}G_{k}(z-\zeta)(\Lambda_{\gamma}-\Lambda_{1})u_1(\zeta,k)dS(\zeta).
\end{equation}
Thus, by \eqref{pnun}, \eqref{jen}, and \eqref{gotit}, we have that $u_1(\cdot,k)|_{\partial\Omega}$ satisfies \eqref{bie2}.
The uniqueness of $u_1(\cdot,k)|_{\partial\Omega}$ follows by Theorem \ref{thm:uexp}.
\end{proof}

An analogous theorem holds for $u_2$.
\begin{theorem}  \label{thm:u2_BIE}  % BIE for u2
Let $\gamma\in W^{2,p}(\Omega)$ for $p>1$ and suppose $\gamma=1$ in a neighborhood of $\partial\Omega.$ Suppose $\sigma$ and $\epsilon$ satisfy \eqref{eq-fran-1-5} and \eqref{eq-fran-1-7}, and let $\gamma(z) - 1$ have compact
support in $W ^{2,p}(\Omega).$  Then for any nonexceptional $k\in \mathbb{C}\setminus\left\{0\right\}$, the trace of the exponentially growing solution $u_2(\cdot,k)$ on $\partial\Omega$ is the unique solution to
\begin{equation}  \label{bie2a}
u_2(z,k)=\frac{e^{-ik\bar{z}}}{-ik}-\int_{\DOm}G_{k}(-\bar{z}+\bar{\zeta})(\Lambda_{\gamma}-\Lambda_{1})u_2(\zeta,k)dS(\zeta),\hspace{0.5 cm} z\in \partial\Omega.
\end{equation}
\end{theorem}

\begin{proof}  %Proof of BIE for u2
Let $p,\;r,\;  \left\{\gamma_n \right\}_{n \in\textbf{N}}\subset W^{2,r}(\Omega),$ and $\psi_n$  be as in the proof of Theorem \ref{thm:u1_BIE}.  Let $u^n$ be the CGO solutions defined in Theorem \ref{thm:uexp2}  to the admittivity equation with admittivity $\gamma_n.$ Then for each $n\in \mathbb{N}$, for nonexceptional $k\in \mathbb{C}\setminus\left\{0\right\}$, evaluating \eqref{psi_N} at $-\bar{z}$,
\begin{equation} \label{psi_BIE_eval}
\psi_n(-\bar{z},k)|_{\partial\Omega}=e^{-ik\bar{z}}|_{\partial\Omega}-\int_{\DOm}G_{k}(-\bar{z}-\zeta)(\Lambda_{\gamma_n}-\Lambda_{1})\psi_n(\zeta,k)dS(\zeta),
\end{equation}
where $\gamma_n=1$ in a neighborhood of $\partial\Omega$.

It follows by \eqref{sa2} that for each complex number $k\neq0,$ and for each $n\in \mathbb{N}$
\begin{equation} \label{pnun2}
\frac{{\gamma_n}^{-1/2}(-\bar{z})}{-ik}\psi_n(-\bar{z},k)=u^n(z,k) \rightarrow u_2(z,k) \hspace{0.3cm} \mbox{in} \hspace{0.4cm} H^{1/2}(\partial\Omega).
\end{equation}

We claim that for each $n,$ $u^n$ satisfies \eqref{bie2a}. To see this, by \eqref{sa2}, for $z\in \DOm$,
\begin{eqnarray}
\frac{e^{-ik\bar{z}}}{-ik} &-& \int_{\DOm}G_{k}(-\bar{z}+\overline{\zeta})(\Lambda_{\gamma}-\Lambda_{1})u^n(\zeta,k)dS(\zeta) \nonumber\\
&=&\frac{e^{-ik\bar{z}}}{-ik}-\int_{\DOm}G_{k}(-\bar{z}+\overline{\zeta})(\Lambda_{\gamma}-\Lambda_{1})\frac{\gamma_n ^{-1/2}(-\overline{\zeta})}{-ik} \psi_n(-\overline{\zeta},k)dS(\zeta) \nonumber\\
&=&\frac{e^{-ik\bar{z}}}{-ik}-\int_{\DOm}G_{k}(-\bar{z}-\tilde{\zeta})(\Lambda_{\gamma}-\Lambda_{1})\frac{\gamma_n ^{-1/2}(\tilde{\zeta})}{-ik} \psi_n(\tilde{\zeta},k)dS(\tilde{\zeta}) \nonumber\\
&=&  \frac{\gamma_n ^{-1/2}(-\bar{z})}{-ik}\psi_n(-\bar{z},k) \nonumber \\
&=&  u^n(z,k),\label{jen2}
\end{eqnarray}
using the change of variables $-\bar{\zeta}\mapsto\tilde{\zeta}$ and the fact that  $\gamma_n =1$ in a neighborhood of $\partial\Omega$.  Thus, $u^n$ satisfies \eqref{bie2a} for each $n\in \mathbb{N}.$

We know by Theorem~3.1 of  \cite{francini} that $M(z,k)$ depends continuously on $\gamma.$  From \eqref{pnun2}, we can conclude that  % Note we are using Lemma 29 from Alan's thesis here, but don't want to write it.
\begin{equation}
\label{gotit2}
\int_{\DOm}G_{k}(-\bar{z}+\bar{\zeta})(\Lambda_{\gamma_n}-\Lambda_{1})u^n(\zeta,k)dS(\zeta) \mapsto \int_{\DOm}G_{k}(-\bar{z}+\bar{\zeta})(\Lambda_{\gamma}-\Lambda_{1})u_2(\zeta,k)dS(\zeta).
\end{equation}
Thus, by \eqref{pnun2}, \eqref{jen2}, and \eqref{gotit2}, we have that $u_2(\cdot,k)|_{\partial\Omega}$ satisfies \eqref{bie2a}.
The uniqueness of $u_2(\cdot,k)|_{\partial\Omega}$ follows by Theorem \ref{thm:uexp2}.
\end{proof}

\subsection{The Scattering Transform Matrix}
The scattering transform $S_{\gamma}(k)$ of the matrix potential $Q_{\gamma}$ is defined in \cite{francini} by
%\begin{equation}
%\label{matrixg}
%S_{\gamma}(k)=\frac {i} {\pi}\int_\mathbf{R^2} \begin{pmatrix} e_{-\bar{k}}(z)  & 0\\ 0  & -e_{k}(z) \end{pmatrix}(Q_{\gamma}M)^{off}(z,k)d\mu(z),
%\end{equation}
%where ``off'' denotes the matrix consisting of only the off-diagonal entries of $Q_{\gamma}M$.  Thus,
\begin{equation}
\label{matrixg}
S_{\gamma}(k)=\frac {i} {\pi}\int_{\R^2} \begin{pmatrix} 0 & Q_{12}(z)e(z,-\bar{k})M_{22}(z,k)   \\  -Q_{21}(z) e(z,k)M_{11}(z,k)  & 0 \end{pmatrix} d\mu(z),
\end{equation}
where $e(z,k)=\exp\{i(zk +\bar{z}\bar{k})\}$.  Thus we are only concerned with computing the off-diagonal entries of $S_{\gamma}$, which we will denote by $S_{12}(k)$ and $S_{21}(k)$.

Boundary integral formulas for the off-diagonal entries of $S_{\gamma}(k)$ in \eqref{matrixg} can be computed by integration by parts as follows
\begin{eqnarray} \nonumber
S_{12}(k) &=& \frac{i}{\pi}\int_{\Omega}Q_{12}(z)e^{-i\bar{k}z}\Psi_{22}(z,k)\;d\mu(z) \\
&=& \frac{i}{\pi}\int_{\Omega}e^{-i\bar{k}z}\dbarz\Psi_{12}(z,k)\;d\mu(z)  \nonumber\\
&=& \frac{i}{2\pi}\int_{\DOm}e^{-i\bar{k}z}\Psi_{12}(z,k) (\nu_1+i\nu_2)\;dS(z)- \frac{i}{\pi}\int_{\Omega}\dbarz e^{-i\bar{k}z} \Psi_{12}(z,k)\;d\mu(z)  \nonumber \\
&=& \frac{i}{2\pi}\int_{\DOm}e^{-i\bar{k}z}\Psi_{12}(z,k) (\nu_1+i\nu_2)\;dS(z) \label{Parts_BIE_for_S12}
\end{eqnarray}
and similarly,
\begin{eqnarray} \label{Parts_BIE_for_S21}
S_{21}(k) = -\frac{i}{2\pi}\int_{\DOm}e^{i\bar{k}\bar{z}}\Psi_{21}(z,k) (\nu_1-i\nu_2)\;dS(z),
\end{eqnarray}
where $\nu=\nu_1+i\nu_2$ denotes the outward unit normal to the boundary $\DOm$.
\begin{theorem}\label{thm-Psi-BIE-Sarah}
The trace of the exponentially growing solutions $\Psi_{12}(z,k)$ and $\Psi_{21}(z,k)$ for $k\in\C\setminus\{0\}$ can be determined by
\begin{eqnarray}
\Psi_{12}(z,k) &=& \int_{\DOm}\frac{e^{i\bar{k}(z-\zeta)}}{4\pi(z-\zeta)}\left[\Lambda_\gamma-\Lambda_1\right]u_2(\zeta,k)\;dS(\zeta)\label{eq-Psi-12-bndry-USE-THIS}\\
\Psi_{21}(z,k) &=& \int_{\DOm}\overline{\left[\frac{e^{i k(z-\zeta)}}{4\pi(z-\zeta)}\right]}\left[\Lambda_\gamma-\Lambda_1\right]u_1(\zeta,k)\;dS(\zeta), \label{eq-Psi-21-bndry-USE-THIS}
\end{eqnarray}
where $u_1$ and $u_2$ are calculated via equations~\eqref{bie2} and \eqref{bie2a} respectively.
\end{theorem}

\begin{proof}
We use the relations in \eqref{e8} and \eqref{e10} to obtain boundary integral equations for $\Psi_{21}$ and $\Psi_{12}$ for $z\in\DOm$ from Equations~\eqref{bie2} and \eqref{bie2a}, respectively.  Let us begin with $\Psi_{12}$:
\begin{eqnarray}
\Psi_{12}(z,k)&=&\gamma^{1/2}(z)\dez u_2(z,k)\nonumber\\
&=&\gamma^{1/2}(z)\dez\left[\frac{e^{-i k\bar{z}}}{-ik}-\int_{\DOm} G_k(-\bar{z}+\bar{\zeta})\left[\Lambda_\gamma-\Lambda_1\right]u_2(\zeta,k)\;dS(\zeta)\right]\nonumber\\
&=&-\gamma^{1/2}(z)\int_{\DOm} \dez\left[G_k(-\bar{z}+\bar{\zeta})\right]\left[\Lambda_\gamma-\Lambda_1\right]u_2(\zeta,k)\;dS(\zeta).\label{eq-Psi-12-bndry-1st-form}
\end{eqnarray}
Similarly,
\begin{eqnarray}
\Psi_{21}(z,k)&=&\gamma^{1/2}(z)\dbarz u_1(z,k)\nonumber\\
&=&\gamma^{1/2}(z)\dbarz\left[\frac{e^{i kz}}{ik}-\int_{\DOm} G_k(z-\zeta)\left[\Lambda_\gamma-\Lambda_1\right]u_1(\zeta,k)\;dS(\zeta)\right]\nonumber\\
&=&-\gamma^{1/2}(z)\int_{\DOm} \dbarz\left[G_k(z-\zeta)\right]\left[\Lambda_\gamma-\Lambda_1\right]u_1(\zeta,k)\;dS(\zeta).\label{eq-Psi-21-bndry-1st-form}
\end{eqnarray}

A thorough study of the properties of the Faddeev Green's function $G_k$ and its derivatives is given in \cite{Samuli_thesis99}.  The calculations for the specific derivatives needed here are shown below.  By the definition of $G_k$ \eqref{g1}
\begin{equation}\label{eq-dez-Gk-1}
\dez G_k(-\bar{z}+\bar{\zeta})= \dez\left[e^{ik(-\bar{z}+\bar{\zeta})}g_k(-\bar{z}+\bar{\zeta})\right]= e^{ik(-\bar{z}+\bar{\zeta})}\dez g_k(-\bar{z}+\bar{\zeta}).
\end{equation}
Using the definition of $g_k$ \eqref{g2},
\begin{eqnarray}
\dez g_k(-\bar{z}+\bar{\zeta})
&=&\dez\frac{1}{(2\pi)^2}\int_{\R^2} \frac{e^{i(-\bar{z}+\bar{\zeta})\cdot\xi}}{\xi(\bar{\xi}+2k)}\;d\xi\nonumber\\
&=&\frac{1}{(2\pi)^2}\int_{\R^2} \frac{\dez\left[e^{-(i/2)(\bar{z}\bar{\xi}+z\xi)}\right]e^{i\bar{\zeta}\cdot\xi}}{\xi(\bar{\xi}+2k)}\;d\xi\nonumber\\
&=&\frac{1}{(2\pi)^2}\int_{\R^2} \frac{(-i\xi/2)\;e^{i(-\bar{z}+\bar{\zeta})\cdot\xi}}{\xi(\bar{\xi}+2k)}\;d\xi\nonumber\\
&=&\frac{e^{i(-\bar{z}+\bar{\zeta})\cdot(-2\bar{k})}}{4}\frac{1}{(2\pi)^2}\int_{\R^2} \frac{2 e^{i(-\bar{z}+\bar{\zeta})\cdot(\xi+2\bar{k})}}{i\left(\overline{\xi+2\bar{k}}\right)}\;d\xi\nonumber\\
&=&\frac{e^{i(-\bar{z}+\bar{\zeta})\cdot(-2\bar{k})}}{4\pi\left(\overline{-\bar{z}+\bar{\zeta}}\right)}\nonumber\\
&=&-\frac{e^{-ik(-\bar{z}+\bar{\zeta})}e^{-i\bar{k}(-z+\zeta)}}{4\pi(z-\zeta)}\label{eq-dez-gk}
\end{eqnarray}
by the definition of the inverse Fourier transform and the well known result 
\[\mathcal{F}^{-1}\left\{\frac{2}{i\bar{\xi}}\right\}^\wedge(z)=\frac{1}{\pi\bar{z}}.\]
Therefore, by \eqref{eq-dez-Gk-1} and \eqref{eq-dez-gk}
\begin{equation}\label{eq-dez-Gk}
\dez G_k(-\bar{z}+\bar{\zeta})= -\frac{e^{i\bar{k}(z-\zeta)}}{4\pi(z-\zeta)}.
\end{equation}
The $\dbarz$ derivative for $\Psi_{21}$ is calculated in a similar manner,
\begin{equation}\label{eq-dbarz-Gk}
\dbarz G_k(z-\zeta)= -\overline{\left[\frac{e^{i k(z-\zeta)}}{4\pi(z-\zeta)}\right]}.
\end{equation}
%
%\begin{eqnarray}
%\frac{\partial}{\partial x} g_k(z) &=&-\frac{1}{4\pi z}-\frac{e(z,-k)}{4\pi\bar{z}}-i k g_k(z)\label{eq-dx-gK}\\
%\frac{\partial}{\partial y} g_k(z) &=&\frac{1}{4\pi i z}-\frac{e(z,-k)}{4\pi i\bar{z}}-k g_k(z)\label{eq-dy-gK},
%\end{eqnarray}
%we arrive at
%\begin{equation}\label{eq-dbarz-gk}
%\dbarz g_k(z-\zeta)=-\frac{e(z-\zeta,-k)}{4\pi(\bar{z}-\bar{\zeta})},
%\end{equation}
%and similarly
%\begin{equation}\label{eq-dez-gk}
%\dez g_k(z-\zeta)=-\frac{1}{4\pi(z-\zeta)}-i k g_k(z-\zeta).
%\end{equation}
%Since $G_k(z)=e^{ikz}g_k(z)$, we use \eqref{eq-dbarz-gk} and \eqref{eq-dez-gk} to compute the $\dez$ and $\dbarz$ derivatives of $G_k(z-\zeta)$, obtaining
%\begin{eqnarray}
%\dez G_k(z-\zeta) &=& \dez\left[e^{i k(z-\zeta)}g_k(z-\zeta)\right]
%%&=& i k e^{i k(z-\zeta)}g_k(z-\zeta)+ e^{i k(z-\zeta)}\dez g_k(z-\zeta)\nonumber\\
%%&=& e^{i k(z-\zeta)}\left[g_k(z-\zeta)-\frac{1}{4\pi(z-\zeta)}-i k g_k(z-\zeta)\right]\nonumber\\
%= -\frac{e^{i k(z-\zeta)}}{4\pi(z-\zeta)},\label{eq-dez-Gk}
%\end{eqnarray}
%and similarly,
%\begin{eqnarray}
%\dbarz G_k(z-\zeta) &=& \dbarz\left[e^{i k(z-\zeta)}g_k(z-\zeta)\right]
%%&=& e^{i k(z-\zeta)}\dbarz g_k(z-\zeta)\nonumber\\
%%&=& e^{i k(z-\zeta)}\left[-\frac{e(z-\zeta,-k)}{4\pi(\bar{z}-\bar{\zeta})}\right]\nonumber\\
%%&=& -\frac{e^{-i\bar{k}(\bar{z}-\bar{\zeta})}}{4\pi(\bar{z}-\bar{\zeta})}\nonumber\\
%= -\overline{\left[\frac{e^{i k(z-\zeta)}}{4\pi(z-\zeta)}\right]}.\label{eq-dbarz-Gk}
%\end{eqnarray}
%Thus,
%\begin{equation}\label{eq-dbarz-Gk-dez-Gk-relation}
%\dbarz G_k(z-\zeta) = \overline{\dez G_k(z-\zeta)}.
%\end{equation}
Substituting the representations for $\dez G_k(-\bar{z}+\bar{\zeta})$ and $\dbarz G_k(z-\zeta)$, given in \eqref{eq-dez-Gk} and \eqref{eq-dbarz-Gk}, back into the equations for $\Psi_{12}$ and $\Psi_{21}$, given in \eqref{eq-Psi-12-bndry-1st-form} and \eqref{eq-Psi-21-bndry-1st-form} respectively, proves the theorem.
\end{proof}

\subsection{From $S(k)$ to $M$}
The dependence of $M$ on the complex parameter $k$ is related to the scattering transform through the following $\dbar_k$ system.
\begin{theorem}[Theorem 4.1 \cite{francini}]\label{thm-fran-4-1}
Let $\sigma$ and $\epsilon$ satisfy \eqref{eq-fran-1-5} and \eqref{eq-fran-1-7} and let $M$ be the unique solution to \eqref{Dbar_z_for_M} satisfying \eqref{Mcondition}.  The map $k\to M(\cdot,k)$ is differentiable as a map into $L^r_{-\beta}$, and satisfies the equation
\begin{equation}\label{eq-dbar-k}
\dbar_k M(z,k)=M(z,\bar{k})\left(\begin{array}{cc}
e(z,\bar{k}) & 0\\
0 & e(z,-k)\\
\end{array}\right) S_\gamma(k),
\end{equation}
where
\begin{equation}\label{eq-fran-4-1}
S_\gamma(k)=\frac{i}{\pi}\int_{\R^2}\left(\begin{array}{cc}
e(z,-\bar{k}) & 0\\
0 & -e(z,k)\\
\end{array}\right)(Q_\gamma M)^\text{off}(z,k)\;d\mu(z).
\end{equation}
Moreover, for every $p>2$,
\[\underset{k}{\sup}\left\|M(z,\cdot)-I\right\|_{L^p(\R^2)}\leq K_2,\]
where $K_2$ depends on $\beta$, $\sigma_0$, $\Omega$, and $p$.
\end{theorem}

Notice that Equation \eqref{eq-dbar-k} can be written as the following two systems of equations:
\begin{equation}\label{eq-dbark-system-1}
\left\{\begin{array}{rc}
\dbar_k M_{11}(z,k)=&M_{12}(z,\bar{k})e(z,-k)S_{21}(k)\\
\dbar_k M_{12}(z,k) =&M_{11}(z,\bar{k})e(z,\bar{k})S_{12}(k)\\
\end{array}\right.,
\end{equation}
and
\begin{equation}\label{eq-dbark-system-2}
\left\{\begin{array}{rc}
\dbar_k M_{21}(z,k)=&M_{22}(z,\bar{k})e(z,-k)S_{21}(k)\\
\dbar_k M_{22}(z,k) =&M_{21}(z,\bar{k})e(z,\bar{k})S_{12}(k)\\
\end{array}\right.,
\end{equation}
included for the reader's convenience.

\subsection{From $M$ to $\gamma$}
\begin{theorem} [Theorem 6.2 \cite{francini}]
For any $\rho>0$,
\begin{equation}\label{eq-M-to-Q-francini}
Q_{\gamma}(z)=\lim_{k_0 \to\infty}\mu(B_\rho (0))^{-1}\int_{\left\{k:\left|k-k_0\right|<\rho\right\}} D_k M(z,k)d\mu(k).
\end{equation}
\end{theorem}
This provides a reconstruction formula for the entries of $Q_{\gamma}$, and one can recover $\gamma$ from $Q_{12}=-\frac{1}{2}\dez\log(\gamma)$ or $Q_{21}=-\frac{1}{2}\dbarz\log(\gamma)$.  However, this formula is computationally impractical as it requires a large $k$ limit of integrals involving $\dbarz$ and $\dez$ derivatives of $M(z,k)$.

We have derived computationally advantageous formulas for recovering the entries of $Q_\gamma$ that only require knowledge of the {\sc CGO} solutions at $k=0$.  Theorem~\ref{thm-M-to-Q-Sarah} provides this direct relation between the {\sc CGO} solutions $M(z,0)$ (from the $\dbar_k$ equation \eqref{eq-dbar-k}) and the matrix potential $Q_\gamma(z)$, eliminating the large $k$ limit required in equation \eqref{eq-M-to-Q-francini} above.

% Sarah's theorem for Q
\begin{theorem}\label{thm-M-to-Q-Sarah}
The entries of the potential matrix $Q_\gamma(z)$ defined in \eqref{matrixim} can be calculated using only knowledge of the {\sc CGO} solutions $M(z,0)$ via
\begin{eqnarray}
Q_{12}(z)&=&\frac{\dbarz M_+(Q_\gamma,z,0)}{M_-(Q_\gamma,z,0)}\label{eq-M-to-Q12}\\
Q_{21}(z)&=&\frac{\dez M_-(Q_\gamma,z,0)}{M_+(Q_\gamma,z,0)}\label{eq-M-to-Q21}.
\end{eqnarray}
where,
\begin{eqnarray}
M_+(Q_\gamma,z,k)&=& M_{11}(Q_\gamma,z,k)+e(z,-k)M_{12}(Q_\gamma,z,k)\label{eq-M+-thm}\\
M_-(Q_\gamma,z,k)&=& M_{22}(Q_\gamma,z,k)+e(z,k)M_{21}(Q_\gamma,z,k)\label{eq-M--thm}.
\end{eqnarray}
\end{theorem}
\begin{proof}
We follow an idea similar to that in \cite{barceloRuiz} and define
\begin{eqnarray}
M_+(Q_\gamma,z,k)&=& M_{11}(Q_\gamma,z,k)+e(z,-k)M_{12}(Q_\gamma,z,k)\label{eq-M+}\\
M_-(Q_\gamma,z,k)&=& M_{22}(Q_\gamma,z,k)+e(z,k)M_{21}(Q_\gamma,z,k)\label{eq-M-}.
\end{eqnarray}
Note that $M_+$ and $M_-$ are only dependent on the $Q_\gamma$ matrix, not $-Q_\gamma^T$ as is required in \cite{barceloRuiz}. Therefore,
\begin{eqnarray*}
\dbarz M_+(Q_\gamma,z,k) &=& Q_{12}(z)e(z,-k)\left[M_-(Q_\gamma,z,k) \right. \\
&=&\qquad + \left. i (k-\bar{k})M_{12}(Q_\gamma,z,k)\right]  \\
\dez M_-(Q_\gamma,z,k) &=& Q_{21}(z)e(z,k)M_+(Q_\gamma,z,k),
\end{eqnarray*}
so that
\begin{eqnarray*}
\dbarz M_+(Q_\gamma,z,0) &=& Q_{12}(z)M_-(Q_\gamma,z,0)\\
\dez M_-(Q_\gamma,z,0) &=& Q_{21}(z)M_+(Q_\gamma,z,0).
\end{eqnarray*}
\end{proof}

One can then reconstruct the $\log$ of the admittivity $\gamma$ from either $Q_{12}$ or $Q_{21}$ by inverting the $\dez$ or $\dbarz$ operator respectively, and exponentiate to recover $\gamma$ explicitly
\begin{equation}
\log(\gamma(z))=-\frac{2}{\pi}\int_{\mathbb{C}}\frac{Q_{12}(\zeta)}{\bar{z}-\bar{\zeta}}\;d\mu(\zeta) =-\frac{2}{\pi}\int_{\mathbb{C}}\frac{Q_{21}(\zeta)}{z-\zeta}\;d\mu(\zeta).  \label{eq-gam-Q21}
\end{equation}

\subsection{The steps of the algorithm}
We now have all the necessary steps for a direct reconstruction algorithm:
\begin{enumerate}
\item Compute the exponentially growing solutions $u_1(z,k)$ and $u_2(z,k)$ to the admittivity equation from the boundary integral formulas \eqref{bie2} and \eqref{bie2a}
\begin{eqnarray}
u_1(z,k)\vert_{\DOm} &=&  \left.\frac{e^{ik{z}}}{ik}\right\vert_{\DOm} - \int_{\DOm}G_k(z-\zeta)(\Lambda_\gamma - \Lambda_1) u_1(\zeta,k) dS(\zeta)\nonumber\\%\label{bie_u_1}\\
u_2(z,k)\vert_{\DOm} &=& \left. \frac{e^{-ik\bar{z}}}{-ik}\right\vert_{\DOm} - \int_{\DOm}G_k(-\bar{z}+\overline{\zeta})(\Lambda_\gamma - \Lambda_1) u_2(\zeta,k) dS(\zeta). \nonumber% \label{bie_u_2}
\end{eqnarray}
\item Compute the off diagonal entries of the {\sc CGO} solution $\Psi(z,k)$ for $z\in\DOm$ from the boundary integral formulas \eqref{eq-Psi-12-bndry-USE-THIS} and \eqref{eq-Psi-21-bndry-USE-THIS}
\begin{eqnarray*}
\Psi_{12}(z,k) &=& \int_{\DOm}\frac{e^{i \bar{k}(z-\zeta)}}{4\pi(z-\zeta)}\left[\Lambda_\gamma-\Lambda_1\right]u_2(\zeta,k)\;dS(\zeta) \label{bie_psi_12}\\
\Psi_{21}(z,k) &=& \int_{\DOm}\overline{\left[\frac{e^{i k(z-\zeta)}}{4\pi(z-\zeta)}\right]}\left[\Lambda_\gamma-\Lambda_1\right]u_1(\zeta,k)\;dS(\zeta). \label{bie_psi_21}
\end{eqnarray*}
\item  Compute the off-diagonal entries of the scattering matrix $S_{\gamma}(k)$ from \eqref{Parts_BIE_for_S12} and \eqref{Parts_BIE_for_S21}
\begin{eqnarray*}
S_{12}(k) &=& \frac{i}{2\pi}\int_{\DOm}e^{-i\bar{k}z}\Psi_{12}(z,k) (\nu_1+i\nu_2)dS(z) \label{bie_S_12}\\
S_{21}(k) &=& -\frac{i}{2\pi}\int_{\DOm}e^{i\bar{k}\bar{z}}\Psi_{21}(z,k) (\nu_1-i\nu_2)dS(z). \label{bie_S_21}
\end{eqnarray*}
\item Solve the $\dbar_k$ equation \eqref{eq-dbar-k} for the matrix $M(z,k)$
\begin{equation*}%\label{eq-dbar-k-again}
\dbar_k M(z,k)=M(z,\bar{k})\left(\begin{array}{cc}
e(z,\bar{k}) & 0\\
0 & e(z,-k)\\
\end{array}\right) S_\gamma(k).
\end{equation*}
\item Reconstruct $Q_{\gamma}$ from Theorem~\ref{thm-M-to-Q-Sarah} and use \eqref{eq-gam-Q21} to compute $\gamma$.
\end{enumerate}

\section{Numerical Implementation}\label{sec:numerics}
In this section, we describe the implementation of the algorithm.  Greater detail of the numerical methods and validations of the computations for  admittivity distributions with twice differentiable real and imaginary parts can be found in \cite{sarah_thesis}, where the solution to the forward problem \eqref{Dbar_z_for_M} is computed and used to validate formulas \eqref{eq-Psi-12-bndry-USE-THIS} an \eqref{eq-Psi-21-bndry-USE-THIS}, as well as computations of the scattering transform.  In this work, we consider examples with discontinuities at the organ boundaries.

\subsection{Computation of the DN map}\label{sec-implement}
An approximation to the DN map was computed by simulating voltage data by the finite element method (FEM),  and then computing a matrix approximation to the map by computing the inner product of the applied currents with the voltages.   This approximation to the DN map has been discussed, for example, in \cite{TMIdata,deangelo,dbar_regul}.  It can be formed analogously in the complex case.

Gaussian white noise was added independently to the real and imaginary parts of the simulated voltages for each current pattern by adding a random vector of amplitude $\eta>0$ multiplied by the maximum voltage value for that current pattern and real or imaginary component to the computed voltages.  We consider noise levels $\eta=0$ and $\eta = 0.0001$, which corresponds to $0.01\%$ noise, the published level of the ACT 3 system \cite{edic}, which applies the trigonometric current patterns used in the simulations here.

\subsection{Computation of the {\sc CGO} solutions and $S_\gamma(k)$}
The {\sc CGO} solutions on the boundary of $\Om$ were computed for each $k$ in a grid $[-K,K]^2$ in the complex plane.  The choice of $K$, which serves as a cut-off frequency, was determined by the behavior of the scattering transforms $S_{12}$ and $S_{21}$.  %This is discussed further in Section  \ref{sec:results}.
As in \cite{dbar_regul} for the D-bar algorithm for conductivity reconstructions, the cutoff frequency $K$ has a regularizing effect, and was chosen here empirically to balance smoothing and numerical error.  We do not address the selection of $K$ by more sophisticated means in this work.

\subsubsection{Computation of $u_1$ and $u_2$}
A boundary integral equation of the form \eqref{bie2} was solved in \cite{deangelo} and \cite{dbar_regul}.  In this work, as in \cite{deangelo}, we employ an approximation to the Faddeev's Green's function $G_k$ that allows for very rapid computation of $u_1$ and $u_2$ from \eqref{bie2} and \eqref{bie2a} respectively.  Namely, $G_k$ is approximated by  the fundamental solution for the Laplacian
\[G_0(z) = \frac{1}{2\pi}\log |z|.\]
Denoting the solutions to \eqref{bie2}, \eqref{bie2a} by $u_1^0$ and $u_2^0$, respectively,  the convolution integrals
\begin{eqnarray}
u_1^0(z,k)\vert_{\DOm} &=&  \left.\frac{e^{ik{z}}}{ik}\right\vert_{\DOm} - \int_{\DOm}G_0(z-\zeta)(\Lambda_\gamma - \Lambda_1) u_1^0(\zeta,k) dS(\zeta)\nonumber\\%\label{bie_u_1}\\
u_2^0(z,k)\vert_{\DOm} &=& \left. \frac{e^{-ik\bar{z}}}{-ik}\right\vert_{\DOm} - \int_{\DOm}G_0(-\bar{z}+\bar{\zeta})(\Lambda_\gamma - \Lambda_1) u_2^0(\zeta,k) dS(\zeta)\nonumber% \label{bie_u_2}
\end{eqnarray}
were computed for $z=z_\ell$, the center of the $\ell$th electrode,  via Simpson's rule, and $G_0$ was set to 0 when $\zeta=z_\ell$.  Note that by the definition of $G_0$, $G_0(z-\zeta)=G_0(-\bar{z}+\bar{\zeta})$.

\subsubsection{Computation of $\Psi_{12}$ and $\Psi_{21}$}
The boundary integral formulas \eqref{eq-Psi-12-bndry-USE-THIS} and \eqref{eq-Psi-21-bndry-USE-THIS} for $\Psi_{12}$ and $\Psi_{21}$,  respectively, require knowledge of $\left[\Lambda_\gamma-\Lambda_1\right]u_j(\zeta,k)$ for $j=1,2$, with $\zeta\in\DOm$, and $k\in\mathbb{C}\setminus\{0\}$.  These values are already computed during the evaluation of $u_1$ and $u_2$ via \eqref{bie2} and \eqref{bie2a}.  Therefore, we merely recall those values and approximate the boundary integral using a finite sum.  One should note that $G_0(z-\zeta)$, $\dbarz G_k(z-\zeta)$, and $\dez G_k(-\bar{z}+\bar{\zeta})$ are all undefined for $z=\zeta$.  We removed these points in the computation by setting their values to zero.

\subsubsection{Computation of the scattering transform}\label{sec-comp-scat}
The off-diagonal entries of the scattering transform matrix, namely $S_{12}(k)$ and $S_{21}(k)$,  were computed inside the square $[-K,K]^2$ (with $k=0$ not included since the formulas for the {\sc CGO} solutions do not hold for $k=0$).
 We compute $S_{12}(k)$ and $S_{21}(k)$ using a finite sum approximation to \eqref{Parts_BIE_for_S12} and \eqref{Parts_BIE_for_S21}:
\begin{eqnarray*}
S_{12}(k) &\approx& \frac{i}{L}\sum_{\ell=1}^L e^{-i\bar{k}z_\ell}\Psi_{12}(z_\ell,k) z_\ell \\
S_{21}(k) &\approx& -\frac{i}{L}\sum_{\ell=1}^L e^{i\bar{k}\overline{z_\ell}}\Psi_{21}(z_\ell,k) \overline{z_\ell},
\end{eqnarray*}
where $z_l$ denotes the coordinate of the $\ell^\text{th}$ equally spaced electrode around $\DOm$ (in this case the unit circle).

\subsection{Solution of the system of D-bar equations}\label{sec-solving-dbarK}
The two systems of $\dbar_k$ equations \eqref{eq-dbark-system-1} and \eqref{eq-dbark-system-2} can be written as the convolutions
\begin{equation}\label{eq-dbark-system-1-convolution}
\left\{\begin{array}{rc}
 1=& M_{11}(z,k)-\frac{1}{\pi k}\ast\left(M_{12}(z,\bar{k})e(z,-k)S_{21}(k)\right)\\
0=&M_{12}(z,k) -\frac{1}{\pi k}\ast\left(M_{11}(z,\bar{k})e(z,\bar{k})S_{12}(k)\right)\\
\end{array}\right.,
\end{equation}
and
\begin{equation}\label{eq-dbark-system-2-convolution}
\left\{\begin{array}{rc}
1 =& M_{22}(z,k)-\frac{1}{\pi k}\ast\left(M_{21}(z,\bar{k})e(z,\bar{k})S_{12}(k)\right)\\
0=&M_{21}(z,k) -\frac{1}{\pi k}\ast\left(M_{22}(z,\bar{k})e(z,-k)S_{21}(k)\right)\\
\end{array}\right..
\end{equation}
A numerical solver for equations of the form
\[\dbar_k v(k)=T(k)\overline{v(k)}.\]
was developed in \cite{FIST}  for the inverse conductivity problem.  The solver is based on the fast method by Vainikko~\cite{vainikko} that uses FFT's for solving integral equations with weakly singular kernels.

In this work, we must solve the systems of equations  \eqref{eq-dbark-system-1-convolution} and \eqref{eq-dbark-system-2-convolution} rather than a single equation.  Furthermore, the unknowns $M(z,k)$ are not conjugated, but instead the argument $k$ is conjugated.  To address this, we interpolated the scattering data $S_\gamma$, computed above in Section~\ref{sec-comp-scat}, to a new $k$-grid that includes
the origin $k=0$ at the center and has an odd number of grid points in both the horizonal and vertical directions. We solve the systems \eqref{eq-dbark-system-1-convolution} and \eqref{eq-dbark-system-2-convolution} on this new $k$-grid using appropriate flip operations to ensure that we access the correct entries in the matrix corresponding to $M(z,\bar{k})$.

To perform the convolution we used Fourier transforms as follows:
\begin{eqnarray*}
\frac{1}{\pi k} &\ast& \left(M_{12}(z,\bar{k})e(z,-k)S_{21}(k)\right)  \\
&= & h_\kappa^2\;\text{IFFT}\left(\text{FFT}\left(\frac{1}{\pi k}\right)\cdot\text{FFT}\left(M_{12}(z,\bar{k})e(z,-k)S_{21}(k)\right)\right),
\end{eqnarray*}
and similarly
\begin{eqnarray*}
\frac{1}{\pi k}&\ast& \left(M_{11}(z,\bar{k})e(z,\bar{k})S_{12}(k)\right) \\
&=& h_\kappa^2\;\text{IFFT}\left(\text{FFT}\left(\frac{1}{\pi k}\right)\cdot\text{FFT}\left(M_{11}(z,\bar{k})e(z,\bar{k})S_{12}(k)\right)\right),
\end{eqnarray*}
where $h_\kappa$ is the step size of the uniform $k$-grid of size 129 $\times$ 129, and $\cdot$ denotes componentwise multiplication.  We used GMRES to solve the resulting linear systems for each value of $z$ in a grid of 128 equally spaced points between [-1.1,1.1] in both the $x$ and $y$ directions and computed $M(z,k)$ for all $|z|\leq1.1$.  The step size in $z$ was $h_z\approx0.0173$.
%To avoid blurring when using the Fourier transforms, we evaluated the Green's functions for the $\dbar_k$ operator on grid twice as large as $M(z,k)$.  We then placed the convolved functions in the centers of matrices on the enlarged grid padded with zeros around the edges.

\subsection{Computation of the admittivity}\label{sec-comp-admittivity}
The admittivity is computed by solving first for $Q_{21}$ from \eqref{eq-M-to-Q21} (note that equivalently one could use  $Q_{12}$ from \eqref{eq-M-to-Q12}), and then solving \eqref{eq-gam-Q21} for
$\log (\gamma)$ in the Fourier domain using FFT's. The functions $M_+$ and $M_-$ in equations
\eqref{eq-M+-thm} and \eqref{eq-M--thm} were evaluated using the entries of $M(z,0)$ recovered when solving the $\dbar_k$ equation (see Section~\ref{sec-solving-dbarK} above).
We used centered finite differences (with a step size of $h_z\approx0.0173$) to evaluate the $\dbarz$ and $\dez$ derivatives of $M_+$ and $M_-$ respectively.  We then performed componentwise division to compute $Q_{12}$ and $Q_{21}$ for $|z|\leq1.1$.  Finally, the computed $\log(\gamma)$ was exponentiated to recover $\gamma$ inside the unit disk.

\section{Numerical Results}\label{sec:results}
We consider several test problems simulating a simplified cross-section of a human torso.  In each example, the admittivity is given by $\gamma = \sigma + i \epsilon$.  That is, the imaginary component includes the temporal angular frequency $\omega$.  Since this is a known value, there is no loss of generality in representing $\gamma$ this way in the simulations.
The complete electrode model (CEM), originally described in \cite{Cheng}, was implemented in the FEM  in order to solve the forward problem. The CEM takes into account both the shunting effect of the electrodes and the contact impedances between the electrodes and tissue.  In our computations, $\Om$ was chosen to be a disk of radius $0.15 m$,  and  the FEM computations were performed on a mesh with 4538 triangular elements and 32 equispaced electrodes $0.029m \times 0.024m$ placed on the boundary. The effective contact impedance was chosen to be $z=0.0057 \Omega m^2$ on all electrodes in our simulations. The current amplitude was chosen to be $C =2 mA $, and the applied current patterns are the trigonometric patterns
\begin{equation}
I^j_\ell=
\begin{cases}
C\cos\left(j\theta_\ell\right), & 1\leq \ell\leq L, \;\;\;1\leq j\leq \frac{L}{2}\\
C\sin\left(\left(\frac{L}{2}-j\right)\theta_\ell\right),& 1\leq \ell\leq L,\;\;\;\frac{L}{2}+1\leq j\leq L-1,
\end{cases}
\end{equation}
where $\theta_\ell=\frac{2\pi\ell}{L}$, $\left|e_\ell\right|$ is the area of the $\ell^{\text{th}}$ electrode, $I_\ell$ is the current on the $\ell^{\text{th}}$ electrode, and $L$ denotes the total number of electrodes.
As in \cite{TMIdata,deangelo}, the currents were normalized to have $\ell^2$-norm of $1$, and the voltages were normalized accordingly.  Also, the DN map was scaled to represent data collected on the unit disk using the relation $\Lambda_{\gamma,1} = r\Lambda_{\gamma,r}$, where the second subscript represents the radius of the disk.

 Where indicated, we added 0.01\% Gaussian relative noise  to the simulated voltages as follows.  Denote the (complex-valued) vector of computed voltage for the $j$-th current pattern by $V^j$,  let $\eta = 0.0001$ denote the noise level, and $N$ a Gaussian random vector (generated by the {\tt randn} commmand in MATLAB) that is unique for each use of the notation $N$.  Denoting the noisy data by $\tilde{V}^j$ we then have $\tilde{V}^j=\mbox{Re}(\tilde{V}^j) +i\;\mbox{Im}(\tilde{V}^j) $ where
\begin{eqnarray*}
\mbox{Re}(\tilde{V}^j) &=& \mbox{Re}(V^j) + \eta \max |\mbox{Re}(V^j)| N \\
\mbox{Im}(\tilde{V}^j) &=& \mbox{Im}(V^j) + \eta \max |\mbox{Im}(V^j)| N.
\end{eqnarray*}
We solve the boundary integral equations \eqref{bie2} and \eqref{bie2a} for the traces of the {\sc CGO} solutions $u_1$ and $u_2$ for $k\in[-K,K]^2$, with $K$ varying for each test problem in this work.   %As in \cite{dbar_regul} for the D-bar algorithm for conductivity reconstructions, the cutoff frequency $K$ has a regularizing effect, and was chosen here empirically to balance smoothing and numerical error.  We do not address the selection of $K$ by more sophisticated means in this work.
The solution $M(z,k)$, to the $\dbar_k$ equation \eqref{eq-dbar-k}, is computed in parallel by the method described in Section \ref{sec-solving-dbarK}.  The low-pass filtering by taking $k\in[-K,K]^2$ results in smooth functions $M_{jp}$, $j,p=1,2$, which are differentiated by centered finite differences to recover $Q_{21}$, as described in Section \ref{sec-comp-admittivity}.
The admittivity $\gamma$ was then computed by \eqref{eq-gam-Q21}.

%Define the relative error of the conductivity, and likewise the permittivity, by
%\begin{equation} \label{rel_error}
%\frac{\|\sigma - \sigma^{(R)}\|_{L^2(\Om)}}{\|\sigma\|_{L^2(\Om)}} \cdot 100\%
%\end{equation}
%and
Define the dynamic range of the conductivity, and likewise the  permittivity, by
\begin{equation} \label{dyn_range}
\frac{\max \sigma^{(K)} - \min \sigma^{(K)}}{\max \sigma  - \min \sigma } \cdot 100\%,
\end{equation}
where the maximum and minimum values are taken on the computational grid for the reconstruction and $\sigma^{(K)}$ denotes the reconstructed conductivity $\sigma$ that was computed using a scattering transform computed on the truncated $k$ grid.

\subsection{Example 1}

The first test problem is an idealized cross-section of a chest with a background admittivity of 1+0i.  We do not include units or frequency in these examples, since our purpose is to demonstrate that the equations in this paper lead to a feasible reconstruction algorithm for complex admittivities.  Reconstructions from more realistic admittivity distributions or experimental data are the topic of future work.  Figure 1 shows the values of the admittivity in the simulated heart and lungs.  Noise-free reconstructions with the scattering transform computed on a $128\times128$ grid for $k\in[-5.5,5.5]^2$  are found in Figure 2.  The reconstruction has a maximum conductivity and permittivity value of $1.1452+0.1802i$, occurring in the heart region and a minimum of $0.8286-0.0247i$, occurring in the lung region, resulting in a dynamic range of $79\%$ for the conductivity and $60\%$ for the permittivity when the negative permittivity value is set to 0.  Although this decreases the dynamic range, we set the permittivity to 0 when it takes on a negative value in any pixel, since physically the permittivity cannot be less than 0.  The reconstruction has the attributes of  good spatial resolution and good uniformity in the reconstruction of the background and its value.

%\begin{table}[ht] \label{table-admitt-1}
%%\caption{Admittivity I} % title of Table
%\centering % used for centering table
%\begin{tabular}{|l|c|c|}
%  \hline
%  & Conductivity $\sigma$ & Permittivity $\epsilon$ \\
%  \hline
%  heart & 1.2 & 0.3 \\
%  lungs & 0.8 & 0.1 \\
%  background & 1&0\\
%  \hline
%\end{tabular}
%\caption{Admittivity I}
%\end{table}

\begin{figure}[h!]
\centering
\hspace{0.5em}
%\subfigure[]{
\includegraphics[width=3.5in]{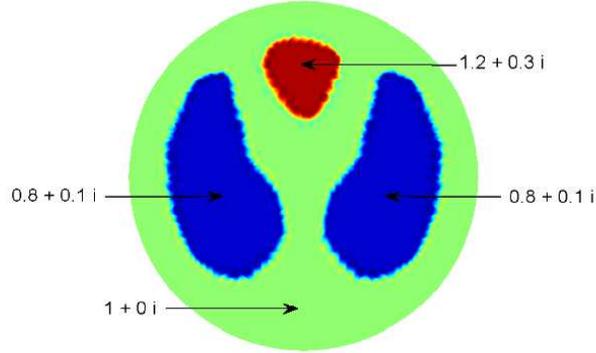}
%}\hspace{2.5em}
%\subfigure[]{
%\includegraphics[width=1.5in]{images_for_paper/test1_permittivity}
%}
\label{fig:Test1}\caption{The test problem in Example 1.}
\end{figure}

\begin{figure}[h!]
\centering
\hspace{0.5em}
%\vspace{2em}
% trim crops left, bottom, right, top
%{\includegraphics[width=6in,  trim=120 290 10 230]{images_for_paper/Ex1_paper_recon.pdf}
{\includegraphics[width=4.8in]{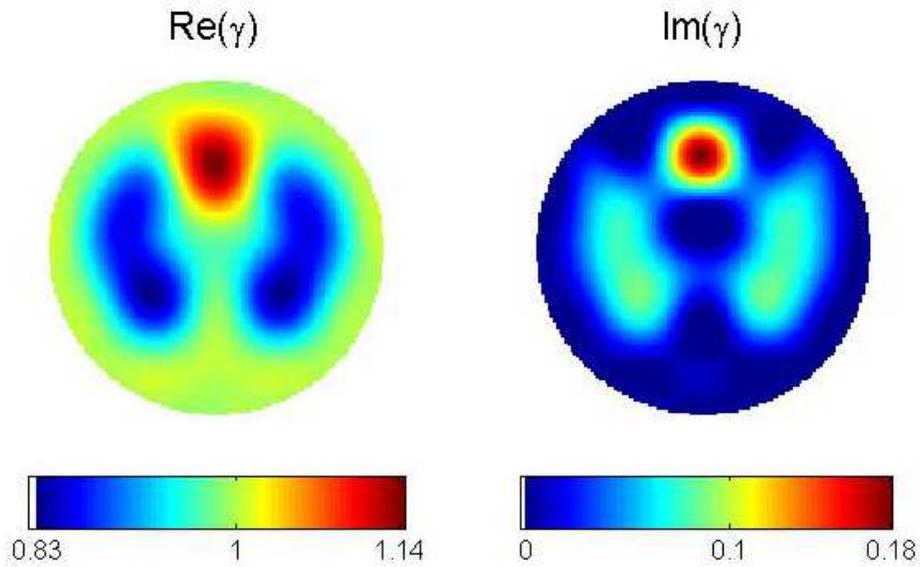}
}
%\subfigure[]{\includegraphics[width=4in]{admitt_1_prescribed_max_min_zero_noise_non_neg_perm}}
%\subfigure[]{\includegraphics[width=4in]{admitt_1_natural_max_min_NOISY}}
%\subfigure[]{\includegraphics[width=4in]{admitt_1_prescribed_max_min_NOISY_non_neg_perm}}
\label{fig:Recons1}\caption{Reconstruction from noise-free data for Example 1 with the real part of $\gamma$ (conductivity) on the left, and the imaginary part (permittivity) on the right.  The cut-off frequency was $K=5.5$.  The dynamic range is 79\% for the conductivity, and 60\% for the permittivity.}
\end{figure}

\subsection{Example 2}

This second example was chosen with conductivity values the same as in Example 1, but with permittivity values in which the ``lungs'' match the permittivity of the background.  This is motivated by the fact that at some frequencies, physiological features may match that of the surrounding tissue in the conductivity or permittivity component.  This example, purely for illustration, mimics that phenomenon.  The admittivity values can be found in Figure 3.   Noise-free reconstructions with the scattering transform computed on a $128\times128$ grid for $k\in[-5.5,5.5]^2$ are found in Figure 4.  The maximum value of the conductivity and permittivity occur in the heart region, $1.1429+0.1828 i$, and the minimum value of the conductivity and permittivity is $0.8271-0.0204 i$.  In this example, the dynamic range is $79\%$ for the conductivity and $61\%$ for the permittivity when the negative permittivity value is set to 0.  Again the spatial resolution is quite good, and the background is quite homogeneous, although some small artifacts are present in both the real and imaginary parts.

%\begin{table}[ht] \label{table-admitt-2}
%%\caption{Admittivity II - Disappearing Lungs} % title of Table
%\centering % used for centering table
%\begin{tabular}{|l|c|c|}
%  \hline
%  & Conductivity $\sigma$ & Permittivity $\epsilon$ \\
%  \hline
%  heart & 1.2 & 0.3 \\
%  lungs & 0.8 & 0 \\
%  background & 1.0 & 0 \\
%  \hline
%\end{tabular}
%\caption{Admittivity II: Disappearing lungs}
%\end{table}

\begin{figure}[h!]
\centering
\hspace{0.5em}
%\subfigure[]{
\includegraphics[width=3.5in]{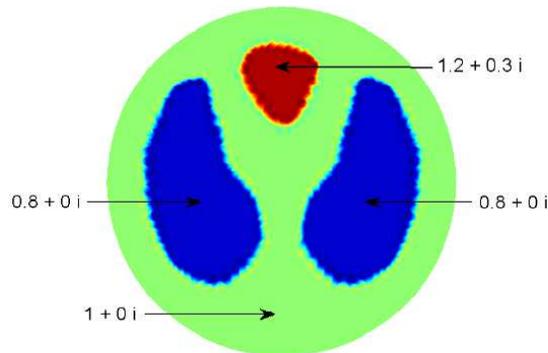}
%}\hspace{2.5em}
%\subfigure[]{
%\includegraphics[width=1.5in]{images_for_paper/test1_permittivity}
%}
\label{fig:Test2}\caption{The test problem in Example 2.   Notice that in this case, the permittivity of the lungs matches the permittivity of the background, and so only the heart should be visible in the imaginary component of the reconstruction.}
\end{figure}

\begin{figure}[h!]
\centering
\hspace{0.5em}
%\vspace{2em}
% trim crops left, bottom, right, top
%{\includegraphics[width=6in, trim=120 290 10 230]{images_for_paper/Ex2_paper_recon.pdf}
{\includegraphics[width=4.8in]{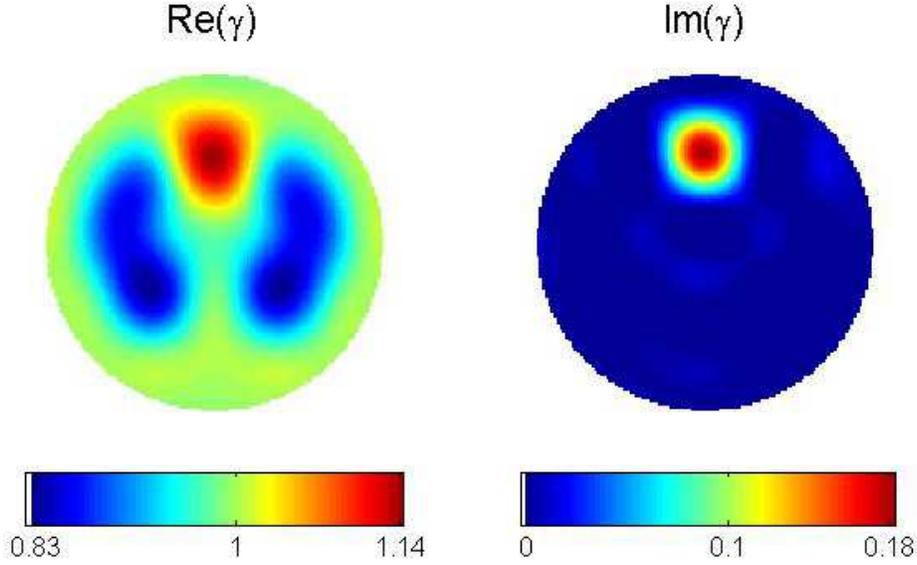}
}
%\subfigure[]{\includegraphics[width=4in]{admitt_2_prescribed_max_min_zero_noise_non_neg_perm}}
%\subfigure[]{\includegraphics[width=4in]{admitt_2_natural_max_min_NOISY}}
%\subfigure[]{\includegraphics[width=4in]{admitt_2_prescribed_max_min_NOISY_non_neg_perm}}
\label{fig:Recons2}\caption{Reconstruction from noise-free data for Example 2 with the real part of $\gamma$ (conductivity) on the left, and the imaginary part (permittivity) on the right.  The cut-off frequency was $K=5.5$.  The dynamic range is 79\% for the conductivity, and 61\% for the permittivity.}
\end{figure}

\subsection{Example 3}

Example 3 is an admittivity distribution of slightly higher contrast, and a non-unitary background admittivity of $\gamma_0=0.8+0.3 i$.  See Figure 5 for a plot of the phantom with admittivity values for the regions.
Due to the non-unitary background, the problem was scaled,  as was done, for example, in \cite{deangelo, TMIdata},  by defining a scaled admittivity $\tilde{\gamma}=\gamma/\gamma_0$ to have a unitary value in the neighborhood of the boundary and scaling the DN map by defining $\Lambda_{\tilde{\gamma}}=\gamma_0 \Lambda_{\gamma}$, solving the scaled problem, and rescaling the reconstructed admittivity.
%The non-unitary background affects equation \eqref{eq-gam-Q21}, which was obtained by the generalized Cauchy integral formula.  In the case of an admittivity equal to a constant $\gamma_0\in\C$ in a neighborhood of the boundary, \eqref{eq-gam-Q21} becomes
%\begin{equation}
%\log(\gamma(z))=\log(\gamma_0) -\frac{2}{\pi}\int_{\mathbb{C}}\frac{Q_{21}(\zeta)}{\bar{z}-\bar{\zeta}}\;d\mu(\zeta).
%\end{equation}
The scattering data for the noise-free reconstruction was computed on a $128\times128$ grid for $k\in[-5.2,5.2]$.   Noisy data was computed as described above in the beginning of this section, and the scattering data was also computed on a $128\times128$ grid for $|k|\leq 5.5$.  The reconstructions are found in Figure 6.  The maximum and minimum values are given in Table 1.  In this example, for the noise-free reconstruction, the dynamic range is $71\%$ for the conductivity and $75\%$ for the permittivity.
Again the spatial resolution is quite good.  There is some degradation in the image and the reconstructed values in the presence of noise. We chose this noise level to be comparable to that of the 32 electrode ACT3 system at RPI \cite{Cook1991}. A thorough study of the effects of noise and stability of the algorithm with respect to perturbations in the data is beyond the scope of this paper.  The scattering transform began to blow up for noisy data, requiring a truncation of the admissible scattering data to a circle of radius 5.5, resulting in a dynamic range of $62\%$  for the conductivity and $68\%$ for the permittivity.  A thorough study of the effects of the choice of $K$ and its method of selection is not included in this paper.

\begin{table}[ht] \label{table-admitt-3}
%\caption{Admittivity III - Non-unitary background} % title of Table
\centering % used for centering table
\begin{tabular}{|c|c|c|c|}
  \hline
  & Admitivity of & Reconstruction from & Reconstruction from \\
  & test problem & noise-free data &noisy data \\
  \hline
  heart & 1.2 + 0.6 i  & 1.0246 + 0.5014 i (max) & 0.9740 + 0.4679 i (max) \\
  lungs & 0.5 + 0.1i  & 0.5262 + 0.1258 i (min) & 0.5390 + 0.1281 i (min)\\
 % background & 0.8+0.3i &  0.8+0.3i &  0.8+0.3i\\
  \hline
\end{tabular}
\caption{Maximum and minimum values in Example 3 with the non-unitary background were found in the appropriate organ region.  The table indicates these values of the admittivity in the appropriate region. }
\end{table}

\begin{figure}[h!]
\centering
\hspace{0.5em}
% trim crops left, bottom, right, top
\includegraphics[width=3.5in]{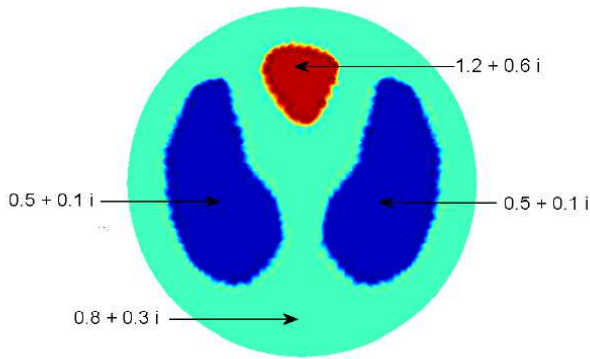}
%\subfigure[]{\includegraphics[width=4in]{admitt_2_prescribed_max_min_zero_noise_non_neg_perm}}
%\subfigure[]{\includegraphics[width=4in]{admitt_2_natural_max_min_NOISY}}
%\subfigure[]{\includegraphics[width=4in]{admitt_2_prescribed_max_min_NOISY_non_neg_perm}}
\label{fig:Test3}\caption{The test problem in Example 3.  In this case, the background admittivity is $0.8+0.3i$, rather than $1+0i$ as in Examples 1 and 2.}
\end{figure}

\begin{figure}[h!]
\centering
\hspace{0.5em}
% trim crops left, bottom, right, top
%\includegraphics[width=6in, trim=120 275 10 250]{images_for_paper/Ex3_paper_recon_no_noise.pdf} \\
%\includegraphics[width=6in, trim=120 275 10 250]{images_for_paper/Ex3_paper_recon_noise.pdf}
\includegraphics[width=4.8in]{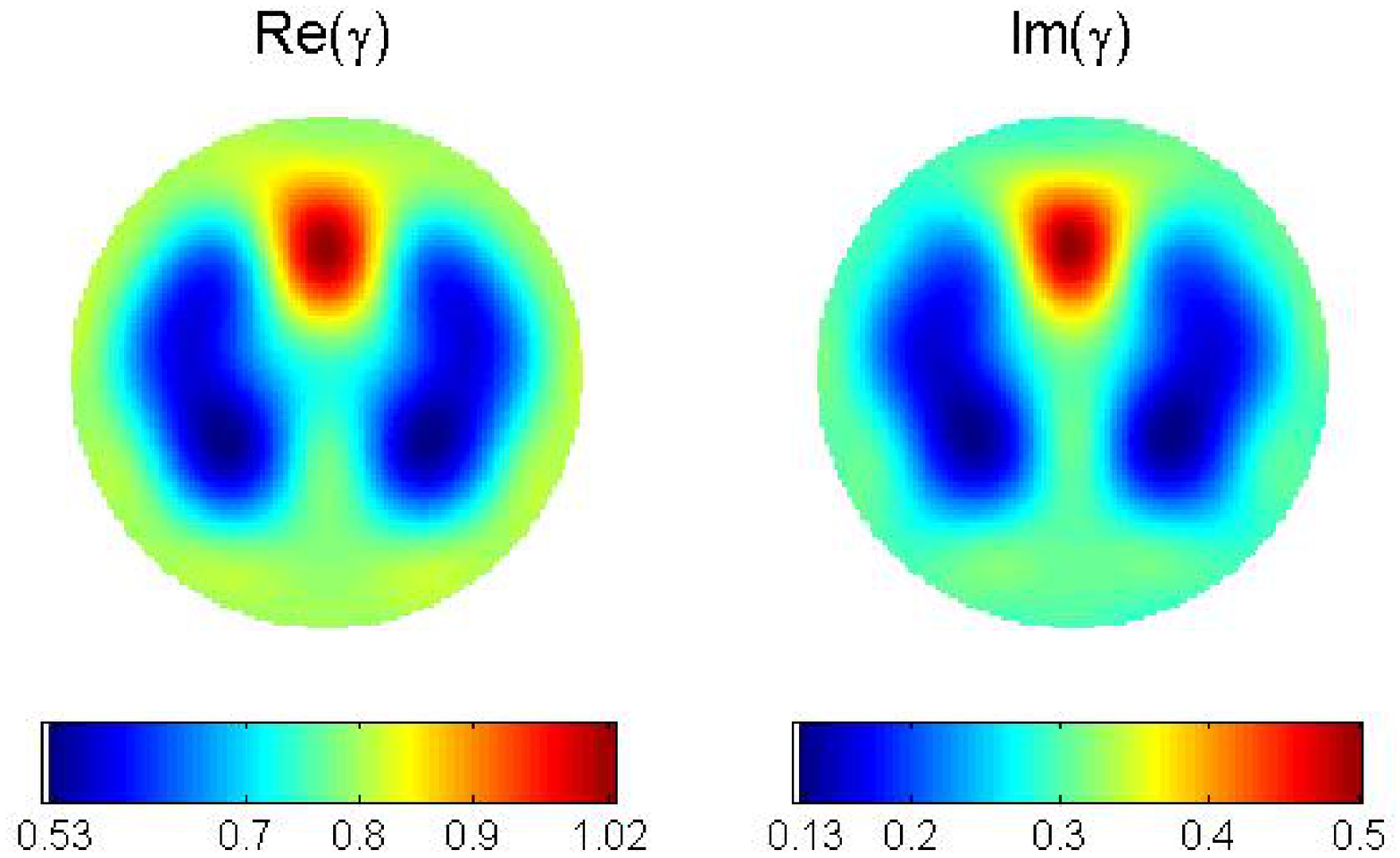} \\
\vspace{2em}
\includegraphics[width=4.8in]{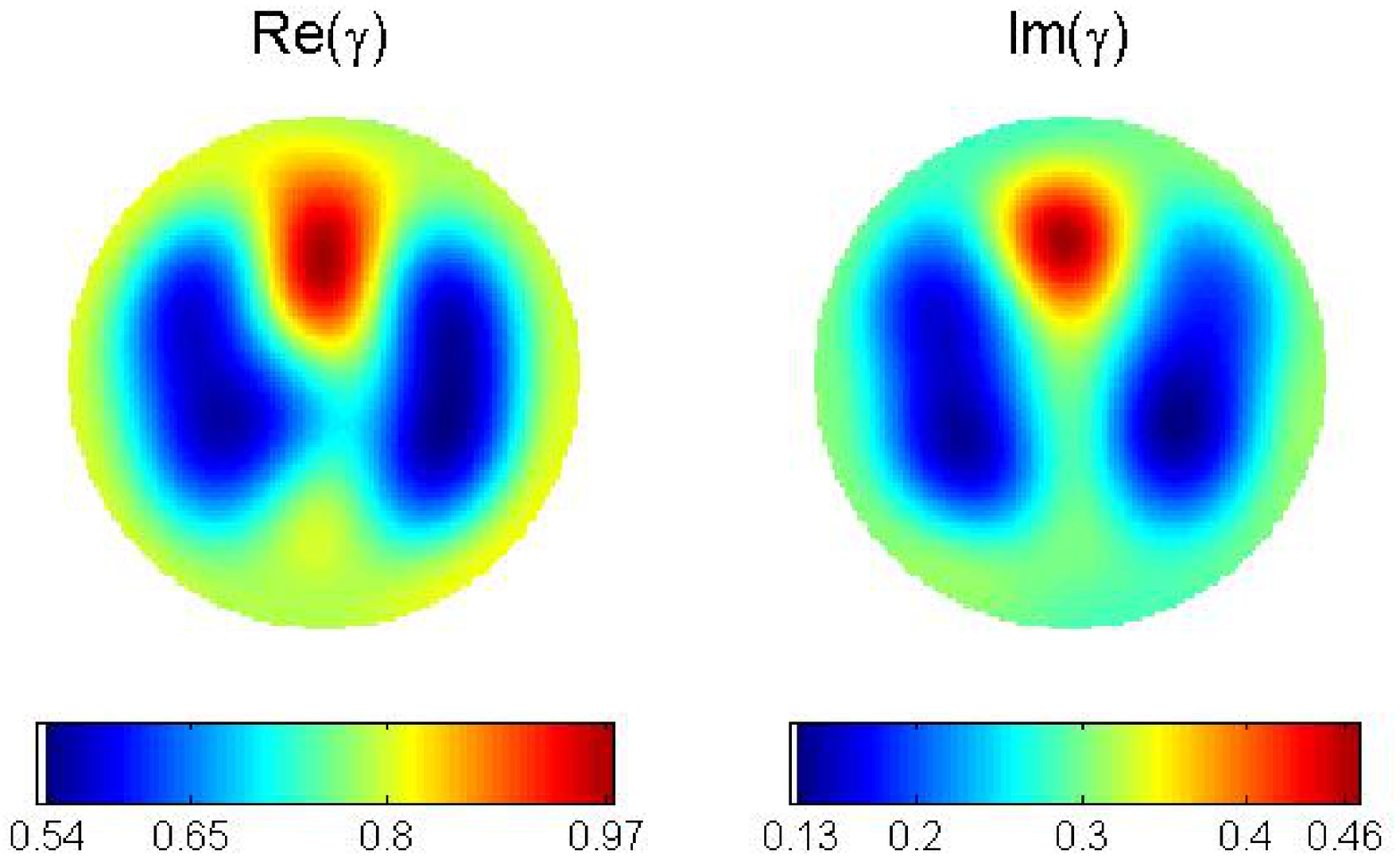}
%\subfigure[]{\includegraphics[width=4in]{admitt_2_prescribed_max_min_zero_noise_non_neg_perm}}
%\subfigure[]{\includegraphics[width=4in]{admitt_2_natural_max_min_NOISY}}
%\subfigure[]{\includegraphics[width=4in]{admitt_2_prescribed_max_min_NOISY_non_neg_perm}}
\caption{\label{fig:Recons2}Top row:  Reconstruction from noise-free data for Example 3.  The cut-off frequency was $K=5.2$.  The dynamic range is 71\% for the conductivity, and 75\% for the permittivity.  Bottom row: Reconstruction from data with $0.01\%$ added noise.  The cut-off frequency was $|k|\leq5.5$.  The dynamic range is 62\% for the conductivity, and 68\% for the permittivity.}
\end{figure}

%\begin{figure}
%\centering
%\hspace{0.5em}
%\subfigure{
%\includegraphics[width=2in]{paper_non_unitary_true_conductivity_scale_max_1o2_min_0o5}
%}\hspace{2.5em}
%\subfigure{
%\includegraphics[width=2in]{paper_non_unitary_true_permittivity_scale_max_0o6_min_0o1}
%}
%\subfigure{\includegraphics[width=6in]{non_unitary_max_min_natural}
%}
%\subfigure{\includegraphics[width=6in]{non_unitary_max_min_true}
%}\label{fig-admitt-3}\caption{Non-unitary test problem}
%\end{figure}

\section{Conclusions}
A new direct method is presented for the reconstruction of a complex conductivity.  This method has the attributes of being fully nonlinear, parallelizable, and the direct reconstruction does not require a high frequency limit.  It was demonstrated on numerically simulated data representing a cross-section of a human chest with discontinuous organ boundaries that the method yields reconstructions with good spatial resolution and dynamic range on noise-free and noisy data.
% Added at the Board Member's recommendation (in our own words)
%Their words:
% Before applying to real data methods derived for solving this mathematical inverse problem, one must remember that such data are noisy, finitely many, and may be insufficient, erroneous and contradictory. Our study with simulated data gives some hope that noise will not be too much perturbing, but much more advanced studies of stability and robustness are necessary to deal with the other defects.
 % My words:
This was the first implementation of such a method, and although efforts were made to realistically simulate experimental data by including discontinuous organ boundaries, data on a finite number of electrodes, and simulated contact impedance, actual experimental data will surely prove more challenging.
 While this study with simulated data gives very promising results, more advanced studies of stability and robustness may be necessary to deal with the more difficult problem of reconstructions from experimental data.

\section{Acknowledgments} The project described was supported by Award Number R21EB009508 from the National Institute Of Biomedical Imaging And Bioengineering.  The content is solely the responsibility of the authors and does not necessarily represent the official view of the National Institute Of Biomedical Imaging And Bioengineering or the National Institutes of Health. C.N.L. Herrera was supported by FAPESP under the process number $2008/08739-7$.

\afterpage{\clearpage}

\end{document}